\documentclass{article}

\usepackage[british]{babel}
\usepackage[cp1250]{inputenc}
\usepackage{amsmath}
\usepackage{amssymb}
\usepackage{amsthm}
\usepackage{verbatim}
\numberwithin{equation}{section}

\newcommand\norm[1]{\left\lVert#1\right\rVert}

\usepackage[toc,page]{appendix}

\usepackage{hyperref}
\usepackage[framemethod=tikz]{mdframed}
\usepackage{amsthm}

\newtheorem{defi}{Definition}[section]
\newtheorem{lem}[defi]{Lemma}
\newtheorem*{asum*}{Assumption}

\newtheorem{rem}[defi]{Remark}
\newtheorem{col}[defi]{Corollary}
\newtheorem{theorem}[defi]{Theorem}
\newtheorem{propo}[defi]{Proposition}

\newtheorem*{theorem*}{Theorem}
\begin{document}
\title{Infinite variance $H$-sssi processes as limits of particle systems}
\author{\L ukasz Treszczotko\footnote{Institute of Mathematics, University of Warsaw, Banacha 2 02-097 Warsaw}\\
\href{mailto:lukasz.treszczotko@gmail.com}{lukasz.treszczotko@gmail.com} }

\maketitle

\begin{abstract}
We consider a particle system with weights and the scaling limits derived from its occupation time. We let the particles perform independent recurrent L\'{e}vy motions and we assume that their initial positions and weights are given by a Poisson point process. In the limit we obtain a number of recently discovered stationary stable self-similar processes studied in~\cite{SAM1} and~\cite{SAM2} as well as a new class of such processes.     
\end{abstract}

\\

\textup{2000} \textit{Mathematics Subject Classification}: \textup{Primary: 60G18} \textup{Secondary: 60F17}

\section{Introduction}

\subsection{The model and the particle system}

The basis of our model is a system of moving particles to which we attach signed weights. Roughly speaking our intention is that the particles are to perform some sort of recurrent motion in $\mathbb{R}$ and that the weights are drawn from a heavy-tailed distribution, so that only non-Gaussian processes may arise as scaling limits of the system in question. The precise description of the particular model variations we intend to study will be given Section~\ref{RS}. 

We consider the following particle system: at time $t\geq 0$ the position of the particles is given by a $(x_j+\eta^j_t)$ and their \emph{weights} are given by $(z_j)$, where $(x_j,z_j)$ are points of a Poisson point process on $\mathbb{R}_+\times \mathbb{R}$ with intensity measure $dt\otimes \nu_{\alpha, \epsilon}(dz)$. $dt$ is the Lebesgue measure and $\nu(dz)$ is a L\'{e}vy measure possessing density which is with exponent $\beta \in (1,2)$. The $\epsilon$ term indicates the fact that what is to follow we discard the weights with absolute value less than $\epsilon$, with $\epsilon$ being an arbitrary positive constant. Finally, we assume that the particles move according to independent symmetric L\'{e}vy processes whose L\'{e}vy measures behave as that of the symmetric $\beta$-stable L\'{e}vy process. For details see the beginning of Section~\ref{RS}.

with $\beta\in (1,2)$ (this will be generalized later on-see the beginning of Section~\ref{RS}). The system of particles at any given point in time $t\geq 0$ is then given by $(x_j+\eta_t^j, z_j)$, where $x_j+\eta_t^j$ represent the the positions of particles and $z_j$'s - their weights . The functional at the centre of our investigations is given by
\begin{equation}\label{FC}
G_t^T:=\frac{1}{F_T}\sum_j z_j\mathbf{1}_{\{|z_j|>\epsilon\}}\int_0^{Tt}\phi(x^j+\eta^j_u)du,
\end{equation}
where $\phi \in L^1(\mathbb{R})$, $T,\epsilon>0$ and $F_T$ is a normalizing constant which will change in different situations. As we shall see, the possible limits, as $T\rightarrow \infty$ of the processes $(G_t^T)_{t\geq 0}$ change drastically depending on whether $\int_\mathbb{R}\phi(y)dy\neq 0$ or not. Moreover, if $\int_\mathbb{R}\phi(y)dy=0$, then the asymptotic behaviour of $\phi$ as $y\rightarrow \infty$ and $y\rightarrow -\infty$ determines which of the number of possible limits are obtained. It will become evident that after appropriate limiting procedures the choice of $\epsilon$ becomes irrelevant. In some sense the behaviour of the limit processes is completely determined by particles with large \emph{weights}.

\subsection{First order asymptotics}

We show that for $\alpha, \beta \in (1,2)$, in the case $\int_\mathbb{R}\phi \neq 0$, the processes $G^T$ given by~\eqref{FC}, with a suitable normalization $F_T$, converge in law to an $\alpha$ stable self-similar process with stationary increments, the so called \emph{$\beta$-stable local time fractional $S\alpha S$ motion}. It has the following integral representation:
\begin{equation}\label{G1}
X=\left(\int_{\mathbb{R}\times \Omega'}L_t(x,\omega')M_\alpha(dx,d\omega')\right)_{t\geq 0},
\end{equation}
where $(L_t(x,\omega'))_{t\geq 0,x \in \mathbb{R}}$ is a jointly continuous version of the local time of the $\beta$-stable L\'{e}vy motion (defined on some probability space $(\Omega',\mathcal{F}',\mathbb{P}')$) and $M_\alpha$ is a symmetric $\alpha$-stable random measure on $\mathbb{R}\times \Omega'$ with control measure $\lambda_1 \otimes \mathbb{P}'$, which is itself defined on some other probability space $(\Omega,\mathcal{F},\mathbb{P})$. Due to the nature of the functional~\eqref{FC}, we will frequently encounter representations of this type in this paper.

In~\cite{DG} the process~\eqref{G1} was obtained in the so called \emph{random rewards schema} (see~\cite{CS}) or \emph{random walks in random scenery models} (see~\cite{DG}). Following~\cite{DG} these models can be described in the following way. Assume that there is a \emph{user} moving randomly on the \emph{network} which earns random rewards (governed by the random scenery) associated to the points in the network that they visit. The quantity of interest is then the total amount of rewards collected. The concrete model considered in~\cite{DG} goes as follows. Assume that that the movement of the user is a random walk on $\mathbb{Z}$ which after suitable scaling converges to the $\beta$-stable L\'{e}vy process with $\beta \in (1,2]$. Furthermore, let the random scenery be given by i.i.d. random variables $(\xi_j)_{j\in \mathbb{Z}}$ which belong to the normal domain of attraction of a strictly stable distribution with index of stability $\alpha \in (0,2]$. Then the \emph{random walk in random scenery} is given by
\begin{equation}
Z_n=\sum_{k=1}^n\xi_{S_k},
\end{equation}
where $S_k=\sum_{j=1}^k X_k$ is the random walk determining the movement of the user. If we consider a large number of independent \emph{random walkers} moving in independent random sceneries, then the scaling limit in the corresponding functional limit theorem (see Theorem 1.2 in~\cite{DG}) leads to the process~\eqref{G1}.

The process~\eqref{G1} was then investigated in~\cite{SAM1} where it arose as a limit of partial sums of a stationary and infinitely divisible process $(X_n)_{n=1}^\infty$ given by
\begin{equation}\label{se1}
X_n=\int_{E}f_n(x)dM(x) 
\end{equation}
where $M$ is a symmetric homogenous infinitely divisible random measure on some measurable space $(E,\mathcal{E})$ with a $\sigma$-finite control measure $\mu$ and local L\'{e}vy measure which is regularly varying at infinity with index $\alpha \in (0,2)$. The $f_n$'s are deterministic functions such that $f_n(x)= f(T^n(x))$ for some ergodic conservative measure preserving map on $(E,\mathcal{E},\mu)$ possessing a Darling-Kac set with a normalizing sequence regularly varying with exponent $\widetilde{\beta} \in (0,1)$. Crucially, it was also assumed that $\int_{E}f(x)\mu(dx)\neq 0$. For details see Theorem 5.1 in~\cite{SAM1} and for general ergodic-theoretical introduction to this setting see Chapter 3 in~\cite{SPLRD}. The parameter $\alpha$ here is the same as in the random walk in random scenery model and the parameters $\beta$ and $\widetilde{\beta}$ satisfy $\beta=1/(1-\widetilde{\beta})$.

\subsection{Second order asymptotics}\label{I2}

The case when the integral $\int_\mathbb{R}\phi(y)dy=0$ is more complicated as the norming $F_T$ and the limit of the process $G^T$ given by~\eqref{FC} depends on finer properties of the function $\phi$. In this case we only consider models with particles moving according to $\beta$-stable L\'{e}vy motions. When $\phi$ has relatively light tails then we show (see Theorem~\ref{THM2}) that $G^T$ converges in the sense of finite-dimensional distributions to a process which belongs to a class of stable self-similar processes recently introduced by Jung, Owada and Samorodnitsky in~\cite{SAM2}. Members of this class have the following integral representation
\begin{equation}\label{G3}
(Y_{\alpha, \beta, \gamma}(t))_{t\geq 0}=\Bigg(\int_{\Omega'\times \mathbb{R}}S_\gamma(L_t(x,\omega'),\omega')dM_\alpha(\omega',x)\Bigg)_{t\geq 0},
\end{equation}
where $(L_t(x))_{t\geq 0}$ is the local time of a symmetric $\beta$-stable L\'{e}vy motion and $S_\gamma$ is an independent symmetric $\gamma$-stable L\'{e}vy motion. Both of these processes are defined on some probability space $(\Omega',\mathcal{F}',\mathbb{P}')$) and $M_\alpha$ is a symmetric $\alpha$-stable random measure on $\Omega'\times\mathbb{R}$ with control measure $\mathbb{P}'\otimes \lambda_1$. Finally, the parameters $\alpha, \beta$ and $\gamma$ satisfy $1<\beta<2$ and $1<\alpha<\gamma<2$.

The model presented in~\cite{SAM2} is basically the same as the one presented in~\cite{SAM1} with one crucial difference being that the function $f$ from the discussion below~\eqref{se1} is such that $\int_{E}f(x)\mu(dx)=0$. Under some  conditions on the function $f$ (see Chapter 4 in~\cite{SAM2}), the limit process of a suitably normalized sequence as in~\eqref{se1} was shown to belong to a class of $H$-sssi stable processes which have an integral representation given by
\begin{equation}\label{G2}
Y_{\alpha,\widetilde{\beta},\gamma}(t):=\int_{\Omega'\times [0,\infty)}S_\gamma(M_{\widetilde{\beta}}((t-x)_+,\omega'),\omega')dZ_{\alpha,\widetilde{\beta}}(\omega',x),\quad t \geq 0,
\end{equation}
where 
\begin{equation*}
0<\alpha<\gamma\leq 2, 0\leq \widetilde{\beta}< 1,
\end{equation*}
$(S_\gamma(t,\omega'))_{t\geq 0}$ is a symmetric $\gamma$-stable L\'{e}vy motion and $(M_{\widetilde{\beta}}(t,\omega'))_{t\geq 0}$ is an independent $\widetilde{\beta}$-Mittag-Leffler process (see section 3 in~\cite{SAM1} for more on the latter). Both of these processes are defined on a probability space $(\Omega',\mathcal{F}',\mathbb{P}')$. Finally $Z_{\alpha,\beta}$ is a $S\alpha S$ random measure on $\Omega' \times [0,\infty)$ with control measure $\mathbb{P}'\otimes \nu_{\widetilde{\beta}}$, where $\nu_{\widetilde{\beta}}(dx)=(1-\widetilde{\beta})x^{-\widetilde{\beta}}\mathbf{1}_{x\geq 0}dx$. By Proposition 3.2 in~\cite{SAM2} the process $Y_{\alpha,\widetilde{\beta},\gamma}$ is $H$-sssi with Hurst coefficient $H=\widetilde{\beta}/\gamma+(1-\widetilde{\beta})/\alpha$. Here we use $\widetilde{\beta}$ instead of $\beta$ so as not to confuse it with the notation we have adopted for this paper. Similarly as in the proof of (3.10) in~\cite{SAM1} we can show that for $\widetilde{\beta}\in (0,\frac{1}{2})$ the process~\eqref{G2} has the same law as~\eqref{G3} with $\beta=(1-\widetilde{\beta})^{-1}$.

The limit process obtained in~\cite{SAM2} corresponds to $\gamma=2$ in~\eqref{G2}. Note that, as far as we know, no relatively natural model is known to yield $Y_{\alpha,\widetilde{\beta},\gamma}$ for $\gamma<2$. In our paper the process given by~\eqref{G3} is obtained as a limit of the functional~\eqref{FC} with $\alpha,\beta \in (1,2)$ and $\gamma=2$ (see Theorem~\ref{THM2}).

In our investigation of the behaviour of the process $G^T$ with $\int_\mathbb{R}\phi(y)dy=0$ we needed some extensions of the results of Rosen in~\cite{ROS} concerning occupation times of stable L\'{e}vy processes. These are given in Section~\ref{WER}. In particular we significantly relax the assumptions made on the function $\phi$ in the original formulation in~\cite{ROS}.\\

By considering the case $\int_\mathbb{R}\phi(y)dy=0$ but with $\phi$ having relatively heavy tails, which are regularly varying at infinity, and taking the limit of the process $G^T$ in~\eqref{FC} we obtain a new class of of self-similar stable processes with stationary increments (see Theorem~\ref{THM3}). We conjecture that similar limits may be obtained in the model considered in~\cite{SAM2} when the assumption (4.7) therein fails.\\

Particle models of this form have proven to be a very fruitful tool for providing representations of self-similar processes with stationary increments. See for example~\cite{BOJTALP} and~\cite{BOJTALIV}.

\subsection{Notation}
Here we fix the notation which we are going to use throughout the rest of the paper. For any $\phi \in L^1(\mathbb{R})$ by $\widehat{\phi}$ we denote its Fourier transform, that is
\begin{equation*}
\widehat{\phi}(z)=\int_\mathbb{R}e^{iuz}\phi(u)du.
\end{equation*}
By $\overset{f.d.d}{\Longrightarrow}$ we denote the convergence of finite-dimensional distributions and by $\overset{\mathcal{C}[0,\infty)}{\Longrightarrow}$ weak convergence in $\mathcal{C}[0,\infty)$. By $RV_0(\delta)$ we denote the set of real-valued functions regularly varying at $0$ with exponent $\delta$. By $c_1,c_2,\ldots$ we denote finite positive constants for which we usually specify the variables on which they depend.

\section{Results}\label{RS}

To state the main results of the paper we first need to provide the assumptions that we make regarding the movements of the particles, their initial positions and the weights they carry throughout their lifetimes.

\begin{asum*}[\textbf{A}]\label{asA}
Let $\eta$ be a L\'{e}vy process with L\'{e}vy measure 
\begin{equation*}
\nu(dx)=c(\beta)^{-1}g(x)dx,
\end{equation*}
where $g$ is a symmetric function regularly varying at infinity with exponent $-1-\beta$ for $\beta \in (1,2)$ and $c(\beta)$ is a positive constant equal to
\begin{equation}
\int_\mathbb{R}\left(1-e^{iu}+iu\mathbf{1}_{\{|u|\leq 1\}}\right)\frac{du}{|u|^{1+\alpha}},
\end{equation}
which guarantees that the rescaled L\'{e}vy exponent of $\eta$ converges to the L\'{e}vy exponent of a symmetric $\beta$ stable L\'{e}vy process with unit scale factor. We can write $g(x)=f(x)|x|^{-1-\beta}$ for $x \in \mathbb{R}$, where $f$ is slowly varying at infinity and symmetric. Let $\eta=(\eta_t)_{t\geq 0}$ be a L\'{e}vy process with characteristic triple $(0,0,\nu)$ and characteristic exponent $\psi$. We always assume that $\psi$ satisfies
\begin{equation}
\int_1^\infty \psi(z)^{-1}dz <\infty.
\end{equation}
\end{asum*}

Note that Assumption (\textbf{A}) is clearly satisfied for symmetric $\beta$-stable L\'{e}vy processes with $\beta \in (1,2)$. Moreover it also admits a larger class of L\'{e}vy processes whose $1$-dimensional distributions are in the domain of attraction of symmetric $\beta$-stable law.

\begin{asum*}[\textbf{B}]\label{asB}
Let $(x_j,z_j)$ be a Poisson point process on $\mathbb{R}^2$ with intensity measure $dx \otimes \nu_{\alpha,\epsilon}(dz)$, where $\nu_{\alpha,\epsilon}(dz):=\mathbf{1}_{\{|z|\geq \epsilon\}}\frac{1}{|z|^{1+\alpha}}L(z)dz$, where $\alpha \in (1,2)$ and $L$ is symmetric and slowly varying at infinity. Assume moreover, that $(\eta^j)$ is a family of i.i.d. L\'{e}vy processes such that $\eta_1$ satisfies Assumption (\textbf{A}). Finally let $\phi$ be any function in $L^1(\mathbb{R})$.
\end{asum*}

\begin{rem}
We discard all points of the random measure with small weights (note the $\mathbf{1}_{\{|z_j|>\epsilon\}}$ term). As we will see later, the particular choice of $\epsilon$ is irrelevant as far as the scaling limits are concerned and the part of the process which we cut off always vanishes in the limit.
\end{rem}

In some cases we will also need additional assumptions, which are stated below.

\begin{asum*}[\textbf{C}]\label{asC}
Assume that there exists $\kappa \in (0,1)$ such that the characteristic exponent from Assumption (\textbf{A}) satisfies
\begin{equation} 
\int_1^\infty \psi(w)^{-\kappa}dw <\infty.
\end{equation}
\end{asum*}

\begin{asum*}[\textbf{D}]\label{asD}
Assume that the function $\phi$ satisfies
\begin{equation}\label{zloc1}
\big|\widehat{\phi}(x+y)-\widehat{\phi}(x)\big|\leq C|y|^{\kappa},
\end{equation}
for all $x, y \in \mathbb{R}$ with $\kappa>(\beta-1)/2$ and that $C$ is a constant independent of $x$ and $y$.
\end{asum*}

\begin{rem}
For~\eqref{zloc1} to hold it suffices to assume that
\begin{equation}
\int_\mathbb{R}|\phi(y)||y|^\kappa<\infty
\end{equation}
for some $\kappa>(\beta-1)/2$.
\end{rem}

\begin{comment}
\begin{rem}
It is not hard to see that~\eqref{zloc1} is satisfied for $\phi$ which behaves in a fairly regular manner. For example, if we assume that the "fatter'' tail of $\phi$ is regularly varying at infinity with exponent $\gamma<-1-(\beta-1)/2$ then $\phi$ satisfies Assumption (\textbf{D}). Moreover, if the "fatter'' tail of $\phi$ can be bounded by a function regularly varying at infinity with exponent smaller tan $-1-(\beta-1)/2$, then $\phi$ also has the property~\eqref{zloc1}.
\end{rem}
\end{comment}

\subsection{Extension of the occupation time limits for stable processes}\label{WER}

First we provide some extensions of already established results which are the building blocks of our main theorems. Those results are as follows (see~\cite{ROS}).

Assume that $(\xi_t)_{t \geq 0}$ is a symmetric $\beta$-stable L\'{e}vy motion with $\beta \in (1,2)$. Then for any $\phi \in L^1(\mathbb{R})$  we have 
\begin{equation}\label{fol1}
\left(T^{\frac{1-\beta}{\beta}}\int_0^{Tt}\phi(\xi_s-T^{1/\beta}x)ds\right)_{t\geq 0} \overset{\mathcal{C}[0,\infty)}{\Longrightarrow}\left(L_t(x)\int_{\mathbb{R}}\phi(y)dy \right)_{t\geq 0},
\end{equation}
where $(L_t(x))_{t\geq 0, x \in \mathbb{R}}$ is a jointly continuous version of a local time of symmetric $\beta$-stable L\'{e}vy process. If $\int_\mathbb{R}\phi(y)dy =0$ then the limit process of left-hand side of~\eqref{fol1} is trivial and a different normalization is more appropriate. In~\cite{ROS}, Rosen proved that if $\phi$ is a bounded Borel function on $\mathbb{R}$ with compact support such that $\int_\mathbb{R}\phi(x)dx=0$, then we have
\begin{equation}\label{ros1}
\frac{1}{T^{\frac{\beta-1}{2\beta}}}\int_0^{Tt}\phi(\xi_s)ds\overset{C([0,\infty))}{\Longrightarrow}\sqrt{d(\phi,\beta)}W_{L_t(0)},
\end{equation}
as $T\rightarrow \infty$, where $W$ is a Brownian motion independent of $\xi$ and $d(\phi,\beta)$ is a constant.\\

The extensions of the above results are given below. As before, $L_t(x)$ stands for the jointly continuous version of the local time of a symmetric $\beta$-stable L\'{e}vy process.

\begin{propo}\label{PROP1}
Assume that $\eta$ is a L\'{e}vy process satisfying Assumption (\textbf{A}) and $\phi \in L^1(\mathbb{R})$ satisfies $\int_\mathbb{R}\phi(y)dy \neq 0$. Then the following convergence holds 
\begin{equation}\label{EQ1}
\left(\frac{1}{F_T}\int_0^{Tf(T^{1/\beta})^{-1}t}\phi(\eta_s-T^{1/\beta}x)ds\right)_{t\geq 0} \overset{f.d.d}{\Longrightarrow}\left(L_t(x)\int_{\mathbb{R}}\phi(y)dy\right)_{t\geq 0},
\end{equation}
as $T\rightarrow \infty$, with $F_T = T^{-1/\beta+1} f(T^{1/\beta})^{-1}$. If, moreover, Assumption (\textbf{C}) holds, then the convergence holds in $\mathcal{C}[0,\infty)$.
\end{propo}

Perhaps more interestingly we prove an extension of the main result of~\eqref{ros1} which greatly relaxes the stringent assumptions made in the original formulation by Rosen in~\cite{ROS}.

\begin{propo}\label{PROP2}
Assume that $\eta$ is a L\'{e}vy process satisfying Assumption (\textbf{A}) and $\phi \in L^1(\mathbb{R})$ satisfies $\int_\mathbb{R}\phi(y)dy = 0$ and Assumption (\textbf{D}). Then the following convergence holds 
\begin{equation}\label{TEQ2}
\left(\frac{1}{F_T^{1/2}}\int_0^{Tf(T^{1/\beta})^{-1}t}\phi(\eta_s-T^{1/\beta}x)ds\right)_{t\geq 0}\overset{f.d.d}{\Longrightarrow} c(\phi)\left(W_{L_t(x)}\right)_{t\geq 0},
\end{equation}
as $T\rightarrow \infty$, and let $F_T = T^{-1/\beta+1} f(T^{1/\beta})^{-1}$, where $W$ is a standard Brownian motion independent of the local time process $(L_t(x))_{x\in \mathbb{R},t\geq 0}$ and 
\begin{equation}
c(\phi)=\frac{1}{\pi}\sqrt{\int_\mathbb{R}\big|\widehat{\phi}(w)\big|^2\psi(w)^{-1}dw}.
\end{equation} 
Moreover, if, additionally, Assumption (\textbf{C}) holds, then the convergence holds in $\mathcal{C}[0,\infty)$.
\end{propo}

This result seems relatively robust, in the sense that we cannot expect Proposition~\ref{PROP2} to hold if the tails of $\phi$ are heavier than $y\mapsto|y|^{-1-(\beta-1)/2}$. If this happens then, at least for $\phi$ with regularly varying tails, the normalization on the left-hand side of~\eqref{TEQ2} is no longer valid and the class of limit processes is different. See section~\ref{FATtails} for deatils.

\subsection{First order limit theorem}\label{RS1}

Here we formulate the first main result of our paper in which we identify the limit process (as $T\rightarrow \infty$) of the functional~\eqref{FC}, provided the function $\phi$ is integrable and the integral $\int_\mathbb{R}\phi(y)dy$ does not vanish.

\begin{theorem}\label{THM1}
Assume that the Assumptions (\textbf{A}) and (\textbf{B}) hold. Consider the functional given by
\begin{equation}\label{GF1}
G_t^T=\frac{1}{N_T}\sum_j z_j\mathbf{1}_{\{|z_j|>\epsilon\}}\int_0^{D_T}\phi(C_Tx^j+\eta_u^j)du, \quad t\geq0,
\end{equation}
where $T\geq 1$, $\epsilon$ is an arbitrary positive constant and let
\begin{eqnarray}
N_T &=& T^{1-1/\beta + 1/\alpha \beta}f(T^{1/\beta})^{-1},\label{N_T}\\
C_T &=& L(T^{1/\alpha \beta}),\\
D_T &:=& T f(T^{1/\beta})^{-1}.
\end{eqnarray}
Then, for any integrable function $\phi$, the process~\eqref{GF1} converges, up to multiplicative constant given by $\int_\mathbb{R}\phi(y)dy$, in the sense of finite dimensional distributions to the \emph{$\beta$-stable local time fractional $S\alpha S$ motion} given by~\eqref{G1}. Furthermore, if (\textbf{C}) holds, then the convergence can be strengthened to weak convergence in $\mathcal{C}[0,\infty)$.
\end{theorem}

\subsection{Second order limit theorems}\label{RS2}

When $\int_\mathbb{R}\phi(y)dy = 0$ the, the limit process given by Theorem~\ref{THM1} is the zero process. To obtain a non-trivial limit in this case one has to use a normalization different than $N_T$ given by~\eqref{N_T}. This case being more complicated, we only consider the case where the particle motion is given by symmetric stable L\'{e}vy processes.

In the case of relatively \emph{light} tails we have the following theorem which produces another representation of the process first described in~\cite{SAM2}.

\begin{theorem}\label{THM2}
Assume that $\eta$ is a symmetric $\beta$-stable L\'{e}vy process with $\beta\in (1,2)$ and in the Assumption (\textbf{B}) the function $L$ is identically equal to $1$. Let $\phi$ be an integrable function with $\int_\mathbb{R}\phi(y)dy = 0$, satisfying Assumption (\textbf{D}) such that additionally
\begin{equation}\label{tloc1}
\int_\mathbb{R}|\phi(y)||y|^{\frac{\beta-1}{2}}dy<\infty.
\end{equation}
Then, the functional given by~\eqref{FC} with $F_T = T^{\frac{\beta-1}{2\beta}+\frac{1}{\alpha \beta}}$ converges, up to multiplicative constant, in the sense of finite dimensional distributions to the process given by
\begin{equation}\label{GW}
\int_{\mathbb{R}\times \Omega'}W_{L_t(x)}M_\alpha(dx,d\omega'), \quad t \geq 0,
\end{equation}
where $M_\alpha$ is a symmetric $\alpha$-stable random measure on $\mathbb{R}\times \Omega'$ with intensity measure $\lambda_1 \otimes \mathbb{P}'$. $(\Omega',\mathcal{F}',\mathbb{P}')$ is the probability space on which $(L_t(x,\omega'))_{t\geq 0, x\in \mathbb{R}}$ is defined. The random measure $M_\alpha$ is itself defined on another probability space $(\Omega,\mathcal{F},\mathbb{P})$. $W$ is a Brownian motion (defined on $(\Omega',\mathcal{F}',\mathbb{P}')$) independent of the movement and the initial positions of the particles. The process defined by~\eqref{GW} is the same as the one in~\eqref{G3} with $\gamma=2$. 
\end{theorem}

\begin{rem}
The assumptions in Theorem~\ref{THM2} regarding the function $\phi$ can be put in a more concise form. For instance it suffices to assume that $\int_{\mathbb{R}}|\phi(y)||y|^\kappa<\infty$ for some $\kappa>(\beta-1)/2$ and $\int_\mathbb{R}\phi(y)dy=0$.
\end{rem}

Things change significantly if we allow $\phi$ to have heavier tails. In this case we have to assume that it is more regular. More precisely, we assume that $\phi$ is regularly varying at $+\infty$ and $-\infty$. We show that in this case the limit process of the functional~\eqref{FC} is a stable $H$-sssi process, which, to our knowledge, has not appeared before.

\begin{theorem}\label{THM3}
Suppose that the particle system and their movements are as in the formulation of Theorem~\ref{THM2} and let $\phi$ be an $L^1(\mathbb{R})$-function such that $\int_\mathbb{R}\phi(y)dy=0$ and
\begin{equation}
\phi(y)=\mathbf{1}_{\{y>0\}}f_1(y)-\mathbf{1}_{\{y<0\}}f_2(-y),
\end{equation}
where the functions $f_1,f_2:(0,\infty)\rightarrow \mathbb{R}$ are integrable, regularly varying at infinity with exponents $-\gamma_1$ and $-\gamma_2$, respectively and such that both $f_1$ and $f_2$ are eventually positive. Then they can be written as $f_1= |\cdot|^{-\gamma_1}g_1$, $f_2= |\cdot|^{-\gamma_2}g_2$ with $g_1$ and $g_2$ being slowly varying at infinity. Furthermore, assume that
\begin{equation*}
\gamma_1,\gamma_2>1, \quad \min(\gamma_1,\gamma_2)<1+\frac{\beta-1}{2}.
\end{equation*}
We consider two possible cases.
\begin{itemize}
\item[(i)]
Assume first that $\gamma_1=\gamma_2=\gamma$ and $\lim_{T\rightarrow \infty}f_1(T)/f_2(T)=1$. Set the normalizing factor
\begin{equation}\label{fat_norm}
F_T=g_1(T^{1/\beta})T^{1+1/(\alpha \beta)-\gamma/\beta}
\end{equation}
in~\eqref{FC}. Then, the process $G^T$, defined by~\eqref{FC} converges, up to a multiplicative constant, in the sense of finite-dimensional distributions to the process which has the following integral representation:
\begin{equation}\label{EQ2}
V_t=\int_{\mathbb{R}\times \Omega'}Z_t(x,\omega')M_\alpha(dx,d\omega'),
\end{equation}
where $M_\alpha$ is a symmetric $\alpha$-stable random measure on $\mathbb{R}\times \Omega'$  with intensity $\lambda_1\otimes \mathbb{P'}$ and
\begin{equation}
Z_t(x,\omega')=\int_0^\infty |y|^{-\gamma}\left(L_t(x+y,\omega')-L_t(x-y,\omega')\right)dy,\quad x\in \mathbb{R},
\end{equation}
with the local time $(L_t(x))$ defined on $(\Omega',\mathcal{F}',\mathbb{P}')$. The random measure $M_\alpha$ itself is defined on another probability space $(\Omega,\mathcal{F},\mathbb{P})$. 
\item[(ii)]
In the second case assume without loss of generality that $\gamma_1<\gamma_2$. Then, the limit process (in the sense of finite-dimensional distributions) of the functional $G^T$ given by~\eqref{FC} with normalization $F_T$ as in~\eqref{fat_norm} has the following integral representation:
\begin{equation}\label{stab2}
\widetilde{V}_t=\int_{\mathbb{R}\times \Omega'}\widetilde{Z}_t(x,\omega')M_\alpha(dx,d\omega'),
\end{equation}
where $M_\alpha$ is as in (i) and
\begin{equation}\label{Z2}
\widetilde{Z}_t(x,\omega')=\int_0^\infty |y|^{-\gamma_1}\left(L_t(x+y,\omega')-L_t(x,\omega')\right)dy,\quad x\in \mathbb{R}.
\end{equation}
\end{itemize}
\end{theorem}
\begin{rem}
In the part (ii) of Theorem~\ref{THM3}, we see that in the limit the heavier tail (corresponding to $\gamma_1$) totally dominates the lighter one (corresponding to $\gamma_2)$, even though the integral of $\phi$ is zero. 
\end{rem}

\begin{propo}\label{PROP3}
The processes $V$ and $\widetilde{V}$ introduced above are $H$-sssi with $H=1+1/(\alpha\beta)-\gamma/\beta$ and $H=1+1/(\alpha\beta)-(\min(\gamma_1,\gamma_2))/\beta$, respectively. In this setting $H$ can take any value between $3/4$ and $1$. 
\end{propo}

\subsection{Organisation of the paper}
The rest of the paper is organised in the following way. In Section~\ref{Psec} we provide some technical results needed to prove our results in full generality. It may be skipped at first reading. In Section~\ref{sec_ros} we provide proof of Propositions~\ref{PROP1} and~\ref{PROP2}. Section~\ref{sec_fol} is devoted to the proof of Theorem~\ref{THM1} and in Section~\ref{sec_sol} we prove Theorems~\ref{THM2} and~\ref{THM3}. Appendix provides some additional technical results which are used throughout the paper.

\section{Technical results related to regular variation}\label{Psec}

This section starts with a few technical results which will be needed to establish Propositions~\ref{PROP1}, ~\ref{PROP2} and Theorem~\ref{THM1} in full generality. Notice that if in the Assumption (\textbf{A}) the process $\eta$ is a symmetric $\beta$-stable L\'{e}vy process, then all the results of subsection~\ref{Psec} become trivial so the reader can skip it if they are interested only in the stable case. 

Throughout this section we let
\begin{eqnarray}
%E_T &:=& T^{1/\beta -1} f(T^{1/\beta})\label{ET}\\
D_T &:=& T f(T^{1/\beta})^{-1},\\
\psi_T(z)&:=& Tf(T^{1/\beta})^{-1}\psi\big(\frac{z}{T^{1/\beta}}\big),
\end{eqnarray}
for $T\geq 1$ and $z\neq 0$. If not stated otherwise, we always assume that $\int_\mathbb{R}|\phi(y)|dy=1$, which implies that $|\widehat{\phi}|$ is bounded by $1$. Note that if $\psi(z)=|z|^\beta$, then $\psi_T(z)=|z|^\beta$ for all $T>0$ and this means that all of the following lemmas become trivial in this case. Before we start we will state a result in the theory of regular variation which will be used multiple times. For proof see Theorem 10.5.6 in~\cite{SPLRD}.

\begin{theorem}\label{RVT}
Let $f$ be a positive function regularly varying at infinity with exponent $\beta\geq -1$. Assume that $f$ is locally integrable, i.e., $\int_0^a f(x)dx<\infty$ for every $0<a<\infty$. Then the function $F(x)=\int_0^xf(t)dt$, $x\geq 0$, is regularly varying at infinity with exponent $\beta+1$ and satisfies
\begin{equation}
\lim_{x\rightarrow \infty}\frac{F(x)}{xf(x)}=\frac{1}{\beta+1},
\end{equation}
with $1/0$ defined as $+\infty$.
\end{theorem}

\begin{lem}\label{clim}
The characteristic exponent from Assumption (\textbf{A}) satisfies
\begin{equation}
\lim_{T\rightarrow \infty}\psi_T(z)=|z|^{\beta},
\end{equation}
for $z \neq 0$.
\end{lem}

\begin{proof}
After a change of variables we can write 
\begin{multline}
\psi_T(z)=c(\beta)^{-1}\int_{\mathbb{R}}\left(1-e^{iuz}+iuz\mathbf{1}_{\{|u|\leq T^{1/\beta}\}}\right)\frac{f(T^{1/\beta}u)}{f(T^{1/\beta})}|u|^{-1-\beta}du,
\end{multline}
with $c(\beta)$ as in Lemma~\ref{trick}. By the same lemma it only remains to justify going with the limit under the integral sign. Fix some $r\in(0,1)$ and write (using the symmetry of $f$) $\psi_T(z) = \psi_T^1(z)+\psi_T^2(z)$ with
\begin{eqnarray}
\psi_T^1(z)&=& c(\beta)^{-1}\int_{|u|\leq r}\big(1-e^{iuz}+iuz\big)\frac{f(T^{1/\beta}u)}{f(T^{1/\beta})}|u|^{-1-\beta}du,\\
\psi_T^2(z)&=& c(\beta)^{-1}\int_{|u|>r}\big(1-e^{iuz}+\mathbf{1}_{\{|u|\leq 1\}}iuz\big)\frac{f(T^{1/\beta}u)}{f(T^{1/\beta})}|u|^{-1-\beta}du\label{kloc1}.
\end{eqnarray}
By Theorem 10.5.6 in~\cite{SPLRD} and inequality $|1-e^{iz}+iz|\leq 1/2|z|^2, z\in \mathbb{R}$, $\psi_T^1(z)$ can be bounded, for all $T$ large enough, by $c_1(\beta)r^{2-\beta}$, where $c_1(\beta)$ is a finite constant depending only on $\alpha$. On the other hand, by Theorem 10.5.5 and Corollary 10.5.8 in~\cite{SPLRD}, for any $\delta>0$ there exists $T_0\geq 1$ such that for all $T\geq T_0$ the integrand in $\psi_T^2(z)$ can be bounded by $c_2|u|^{-\beta}(1+\delta)|u|^{\delta}$. Thus, by dominated convergence, ~\eqref{kloc1} converges to 
\begin{equation}
c(\beta)^{-1}\int_{|u|>r}\big(1-e^{iuz}+\mathbf{1}_{\{|u|\leq 1\}}iuz\big)|u|^{-1-\beta}du.
\end{equation}
This and Lemma~\ref{trick} shows that for any $\delta>0$ and $z \neq 0$ $|\psi_T(z)-|z|^\beta|<\delta$ for all $T$ large enough.
\end{proof}

\begin{comment}
%\begin{rem}
%The assumption~\eqref{clim2} in Lemma~\ref{clim} is not necessary if, for instance, we assume that $f$ is non-decreasing.
%\end{rem}
If, for any $x\in \mathbb{R}$,  we consider the process
\begin{equation}\label{clim6}
P^T_t(x):= \frac{1}{F_T}\int_0^{Tf(T^{1/\beta})^{-1}}\phi(\eta_s-T^{1/\beta}x)ds, \quad t\geq 0,
\end{equation}
with $F_T=T^{1-1/\beta}f(T^{1/\beta})^{-1}$, then we can easily show (using Fourier transform) that the moments of the process given by~\eqref{clim6} converge to the moments of the process defined by the right-hand side of~\eqref{fol1}. Using similar methods one can easily show tightness and conclude that the following generalization of~\eqref{fol1} holds:
\begin{equation}
\left(P^T_t(x)\right)_{t\geq 0} \overset{\mathcal{C}[0,\infty)}{\Rightarrow}\left(L_t(x)\int_{\mathbb{R}}\phi(y)dy\right)_{t\geq 0},
\end{equation}
for any $x\in \mathbb{R}$ and $\phi \in L^1(\mathbb{R})$.
\end{comment}

\begin{col}\label{RVpsi}
It is easy to see that Lemma~\ref{clim} implies that
\begin{equation}
\lim_{w\rightarrow 0}\frac{\psi(w)}{|w|^\beta f(1/w)}=c_1,
\end{equation}
for some finite constant $c_1$. This in turn means that $\psi \in RV_0(\beta)$ and, since we can always write $\psi(z) = |z|^\beta L_0(z)$ with $L_0$ slowly varying at $0$, we have
\begin{equation}
\lim_{T\rightarrow \infty} \frac{L_0(T^{-1})}{f(T)} = 1.
\end{equation}
Moreover, for any $u\geq 0$ and $\theta\in \mathbb{R}$ we have
\begin{equation}\label{LSCAL}
\lim_{T\rightarrow \infty} \mathbb{E}\bigg(exp\big(i\theta T^{-1/\beta}\eta_{D_Tu}\big)\bigg) = e^{-u|\theta|^{\beta}}.
\end{equation}

\end{col}

\begin{lem}\label{LE2}
Let $\psi$ be a L\'{e}vy exponent satisfying
\begin{equation}
\int_1^\infty \psi(z)^{-1}dz <\infty,
\end{equation}
and
\begin{equation}
\psi \in RV_{0}(\beta), \quad \beta \in (1, 2).
\end{equation}
Then for any $K>0$ and $\epsilon>0$ there exists $T_0\geq 1$ and $C_0>0$, such that for all $T\geq T_0$
\begin{equation}
\int_K^\infty \psi_T(z)^{-1}dz \leq C_0 K^{1-\beta} + \epsilon.
\end{equation}
\end{lem}

To prove the above lemma we will need the following consequence of Theorem 10.5.6 in~\cite{SPLRD}.

\begin{lem}\label{LE3}
Let $h: (0,\infty) \rightarrow (0,\infty)$ be in $RV_0(\beta)$, with $\beta \in (1,2)$. Then the function
\begin{equation}
w \mapsto \int_w^1 h(z)^{-1}dz, \quad w \in (0,1)
\end{equation}
is in $RV_0(1-\beta)$ and 
\begin{equation}
\lim_{w\rightarrow 0} \frac{\int_w^1 h(z)^{-1}dz}{wh(w)^{-1}} = \frac{1}{\beta - 1}.
\end{equation}
\end{lem}

\begin{proof}[Proof of Lemma~\ref{LE3}]
Changing variables we have
\begin{equation}
\int_w^1 h(z)^{-1}dz = \int_1^{\frac{1}{w}}h(1/z)^{-1}z^{-2}dz
\end{equation}
and the function $z \mapsto h(1/z)^{-1}z^{-2}$ is in $RV_\infty(\beta-2)$, so by Theorem 10.5.6 in~\cite{SPLRD} the function $x \mapsto \int_1^x h(1/z)^{-1}z^{-2}$ is in $RV_\infty(\beta - 1)$ and
\begin{equation}
\lim_{x \rightarrow \infty} \frac{\int_1^x h(1/z)^{-1}z^{-2}dz}{xh(1/x)^{-1}x^{-2}} = \frac{1}{\beta - 1},
\end{equation}
which finishes the proof of the lemma.
\end{proof}

\begin{proof}[Proof of Lemma~\ref{LE2}]
Notice that we can write
\begin{equation}
\int_K^\infty \psi_T(z)^{-1}dz = A_T(K) + B_T(K),
\end{equation}
where
\begin{eqnarray}
A_T(K) &=& E_T\int_{KT^{-1/\beta}}^1 \psi(z)^{-1}dz,\\
B_T(K) &=& E_T\int_1^\infty \psi(z)^{-1}dz,
\end{eqnarray}
with $E_T = T^{1/\beta -1} f(T^{1/\beta})$.  For $T$ sufficiently large $B_T(K)$ can be made arbitrarily small, irrespective of the value of $K$. By Lemma~\ref{LE3} 
\begin{equation}
\lim_{T\rightarrow \infty} \frac{\int_{KT^{-1/\beta}}^1 \psi(z)^{-1}dz}{KT^{-1\beta}\psi(KT^{-1/\beta})^{-1}} = \frac{1}{\beta-1},
\end{equation}
which means that for $T$ sufficiently large
\begin{equation}\label{local1}
A_T(K) \leq c(\beta, K) K^{1-\beta} \frac{f(T^{1/\beta})}{L_0(KT^{-1/\beta})},
\end{equation}
where the fraction on the right-hand side of~\eqref{local1} converges to $1$ as $T\rightarrow \infty$ by Corollary~\ref{RVpsi} and $c(\beta, K)$ is a finite positive constant independent of $T$.
\end{proof}

\begin{comment}
\begin{lem}\label{LE5}
With our standing assumptions on $\psi$, for any $\kappa > \beta-1$ we have the following inequality:
\begin{equation}
I_T :=\int_0^t \int_\mathbb{R}\left(\big|wT^{-1/\beta}\big|^\kappa \wedge 1 \right)e^{-u\psi_T(w)}dw du \leq c(\kappa, \beta, t)E_T,
\end{equation}
for $t \geq 0$, $T\geq 1$. 
\end{lem}

\begin{proof}
We may write $I_T = I_1 + I_2$, where
\begin{eqnarray}
I_1 &=& \int_0^t \int_{|w|\leq T^{1/\beta}} |wT^{-1/\beta}|^\kappa e^{-u\psi_T(w)}dw du \\
&=& T^{1/\beta}\int_0^t \int_{|w|\leq } |w|^\kappa e^{-u D_T\psi(w)}dw du \\
&\leq& T^{1/\beta}D_T^{-1} \int_{|w|\leq 1} |w|^\kappa \psi(w)^{-1}dw
\end{eqnarray}
Since $\psi \in RV_0(\beta)$ and $\kappa>\beta - 1$, $\int_{|w|\leq 1} |w|^\kappa \psi(w)^{-1}dw$ is finite. On the other hand,
\begin{eqnarray}
I_2 &=& \int_0^t \int_{|w|> T^{1/\beta}} e^{-u\psi_T(w)}dw du\\
&=& T^{1/\beta}D_T^{-1}\int_0^{tD_T} \int_{|w|> 1}  e^{-u\psi(w)}dw du \\
&\leq& E_T\int_{|w|>1}\psi(w)^{-1}dw.
\end{eqnarray}
Thus, the proof is finished.
\end{proof}
\end{comment}

\begin{lem}\label{LE5'}
\begin{comment}
Assume that Assumption (\textbf{A}) holds. Then for any $t\geq 0$ 
\begin{equation}
I_T':= \int_0^{D_T t} \int_{\mathbb{R}}e^{-u\psi_T(z)}dz du \leq C_1(\psi,t)E_T^{-1},
\end{equation}
for some finite constant $c(\psi,t)$ depending on $\psi$ and $t$.
\end{comment}
For any $t>0, T\geq 1$
\begin{equation}
\int_0^t \int_{\mathbb{R}}e^{-u\psi_T(z)}dz du \leq C_2(\psi,t),
\end{equation}
and the constant $C_2(\psi,t)$ does not depend on $T\geq 1$.
\end{lem}

\begin{proof}
This is an easy consequence of Lemma~\ref{LE2}.
\end{proof}

If we additionally assume (\textbf{C}) then we can rephrase Lemma~\ref{LE5'} to obtain the following.

\begin{lem}\label{LTIGHT1}
Suppose that assumptions (\textbf{A}) and (\textbf{C}) are satisfied. Then there exists a constant $c_0$, independent of $t$ and $T$ such that for all $T$ large enough
\begin{equation}\label{nloc1}
\int_0^t \int_\mathbb{R} e^{-u\psi_T(w)}dwdu \leq c_0t^{\delta},
\end{equation}
for any $0<\delta<1-1/\beta$ and all $t\in (0,1)$.
\end{lem}

\begin{proof}
Notice that, by Lemma~\ref{clim}
\begin{equation}
\lim_{w\rightarrow 0}\frac{\psi(w)}{|w|^\beta f(1/w)}=c_1,
\end{equation}
for some finite constant $c_1$. Thus, there exists $\epsilon_1>0$ such that
\begin{equation}\label{nloc0}
\frac{1}{2}c_1 \leq \frac{\psi(w)}{|w|^\beta f(1/w)} \leq 2c_1,
\end{equation}
for $|w|<\epsilon_1$. We may write the left-hand side of~\eqref{nloc1} as $I_1$ + $I_2$, with
\begin{eqnarray}
I_1 &=& \int_0^t \int_{|wT^{-1/\beta}|>\epsilon_1} e^{-u\psi_T(w)}dwdu,\\
I_2 &=& \int_0^t \int_{|wT^{-1/\beta}|\leq \epsilon_1} e^{-u\psi_T(w)}dwdu.\\
\end{eqnarray}
Let us consider $I_1$ first. Since for any $t>0$ and $x>0$
\begin{equation*}
\left|\frac{1-e^{-tx}}{x}\right|\leq min(t,1/x),
\end{equation*}
we have that, in particular, for $\kappa \in (0,1)$
\begin{equation*}
\left|\frac{1-e^{-tx}}{x}\right|\leq t^\kappa x^{\kappa-1}.
\end{equation*}
Therefore, for all $T$ large enough,
\begin{eqnarray*}
I_1 &\leq& \int_{|wT^{-1/\beta}|>\epsilon_1}\Big(\frac{1}{\psi_T(w)}\Big)^{1-\kappa}t^\kappa dw \\
&=& \int_{|w|>\epsilon_1}\Big(\frac{1}{D_T\psi(w)}\Big)^{1-\kappa}T^{1/\beta}t^\kappa dw \leq c_2t^\kappa
\end{eqnarray*} 
for some constant $c_2$ independent of $t$ and $T$. As for $I_2$, one can easily deduce from~\eqref{nloc0} that for $|wT^{-1/\beta}|\leq \epsilon_1$
\begin{equation*}
\psi_T(w)\geq \frac{1}{2}|w|^{\beta}\frac{f(T^{1/\beta}/w)}{f(T^{1/\beta})}.
\end{equation*} 
Fix any $\epsilon_2>0$. An application of Karamata's representation theorem (see for example Theorem 10.5.7 in~\cite{SPLRD}) yields the inequality
\begin{equation}
\frac{f(T^{1/\beta}/w)}{f(T^{1/\beta})}\geq c_3|w|^{-\epsilon_2}
\end{equation}
for all $T$ large enough, $|w|>1$ and $|wT^{-1/\beta}|\leq \epsilon_1$, provided we choose $\epsilon_1$ small enough. $c_3$ is a positive constant independent of $T$. Using all this we may write
\begin{eqnarray*}
I_2 &\leq&  \int_{1<|w|\leq T^{1/\beta}\epsilon_1}\int_0^t e^{-\frac{1}{2}u c_1c_3 |w|^{\beta-\epsilon_2}}dudw + t\int_{|w|<1}dw \\
&\leq& c_4 t^{1-(\beta-\epsilon_2)^{-1}} + 2t,
\end{eqnarray*}
provided we choose $\epsilon_2$ small enough. Thus, we can take
\begin{equation}\label{deltat}
\delta = 1-(\beta-\epsilon_2)^{-1}
\end{equation}
and the proof is finished since $\beta\in (1,2)$.
\end{proof}

\begin{lem}\label{mloc3}
Assume that (\textbf{A}) and (\textbf{D}) hold. The for any $t>0, \kappa>0$ and all $T$ sufficiently large we have the following inequalities:

\begin{equation}\label{mloc3_1}
\int_\mathbb{R} \int_0^t \Big(1\wedge \Big|\frac{w}{T^{1/\beta}}\Big|^{\kappa}\Big)e^{-u\psi_T(w)}dudw \leq c_1(\psi,\beta,\kappa,t)(T^{1/\beta}D_T^{-1}+T^{-\kappa/\beta}),
\end{equation}

\begin{equation}\label{mloc3_4}
\int_\mathbb{R} \int_0^{D_Tt}\Big(1\wedge |w|^{\kappa}\Big)e^{-u\psi(w)}dudw \leq c_4(\psi,\beta,\kappa,t)(1+T^{-1/\beta-\kappa/\beta}D_T).
\end{equation}

In particular for $\kappa>(\beta-1)/2$
\begin{equation}\label{mloc3_2}
\int_\mathbb{R} \int_0^t \Big(1\wedge \Big|\frac{w}{T^{1/\beta}}\Big|^{2\kappa}\Big)e^{-u\psi_T(w)}dudw \leq c_2(\psi,\beta,\kappa,t)T^{1/\beta}D_T^{-1}.
\end{equation}
Furthermore
\begin{equation}\label{mloc3_3}
\int_\mathbb{R} \int_0^{D_Tt} e^{-u\psi(w)}dudw \leq c_3(\psi,\beta,t)T^{-1/\beta}D_T,
\end{equation}

for some finite constants $c_1,c_2,c_3$ and $c_4$ independent of $T$.
\end{lem}

\begin{proof}
After a change of variables $w'=T^{1/\beta}w$, $u' = u/D_T$, the left-hand side of~\eqref{mloc3_1} can be written as
\begin{eqnarray}
&&T^{1/\beta}D_T^{-1}\int_\mathbb{R} \int_0^{D_Tt} \Big(1\wedge |w|^{\kappa}\Big)e^{-u\psi(w)}dudw\label{mloc4} \\
&\leq& T^{1/\beta}D_T^{-1}\Big(\int_{|w|>1}\psi(w)^{-1}dw + \int_{|w|\leq 1}\int_0^{D_Tt}|w|^\kappa e^{-u\psi(w)}dwdu\Big) \label{mloc4'}
\end{eqnarray}
The first integral in~\eqref{mloc4'} is bounded by Assumption (\textbf{A}). We can bound the second by
\begin{equation*}
\int_{|w|\leq T^{-1/\beta}}D_T|w|^\kappa dw + \int_{T^{-1/\beta}<|w|\leq 1}|w|^\kappa\psi(w)^{-1}dw.
\end{equation*}
An application of Lemma~\ref{LE3} gives inequality~\eqref{mloc3_1} and~\ref{mloc3_4}. Inequality~\eqref{mloc3_2} follows immediately once we make a change of variables and use the fact that  $\int_\mathbb{R}\big(1\wedge |w|^{2\kappa}\big)\psi(w)^{-1}dw<\infty$. The inequality~\eqref{mloc3_3} is just Lemma~\ref{LE5'} after a change of variables. 
\end{proof}

\section{Proofs for Section~\ref{WER}}\label{sec_ros}

\subsection{Proof of Proposition~\ref{PROP1}}\label{Psec2}

\begin{comment}
\begin{propo}\label{LE6}
Put
\begin{equation}
P^T_t(x):= \frac{1}{F_T}\int_0^{Tf(T^{1/\beta})^{-1}}\phi(\eta_s-T^{1/\beta}x)ds, \quad t\geq 0,
\end{equation}
with $F_T = T^{-1/\beta+1} f(T^{1/\beta})^{-1} = E_T^{-1}$. Then, for any $\phi \in L^1(\mathbb{R})$, assuming that Assumption (\textbf{A}) holds, we have
\begin{equation}\label{PCONV}
\left(P^T_t(x)\right)_{t\geq 0} \overset{f.d.d.}{\Longrightarrow} \left(L_t(x) \int_\mathbb{R}\phi(y)du\right)_{t\geq 0}.
\end{equation}
Moreover, if we make the assumption C below then the convergence holds in $\mathcal{C}[0,\infty)$.
\end{propo}
\end{comment}

The proof is relatively straightforward once we use the Fourier transform to show the convergence of appropriate moments, therefore we only give a short sketch.

\begin{proof}[Sketch of proof of Proposition~\ref{PROP1}]
Let us put
\begin{equation}\label{P_def}
P_t^T (x) = \frac{1}{F_T}\int_0^{Tf(T^{1/\beta})^{-1}t}\phi(\eta_s-T^{1/\beta}x)ds,
\end{equation}
for $T,t \geq 0$, with $F_T = T^{1-1/\beta}f(T^{1/\beta})^{-1}$. An easy application of Plancherel formula and change of variables formula shows that for any positive integer $k$
\begin{multline}\label{loc2}
\mathbb{E}\big(P_t^T(x)^k\big) = \\
k!\bigg(\frac{1}{2\pi}\bigg)^k \int_{0<s_1 < \ldots < s_k <t}\int_{\mathbb{R}^k}\widehat{\phi}\bigg(\frac{w_1 - w_2}{T^{1/\beta}}\bigg)\widehat{\phi}\bigg(\frac{w_2 - w_3}{T^{1/\beta}}\bigg)\\
\times\ldots \times \widehat{\phi}\bigg(\frac{w_{k-1} - w_k}{T^{1/\beta}}\bigg)\widehat{\phi}\bigg(\frac{w_k}{T^{1/\beta}}\bigg)\\
\times e^{iw_1x}e^{-s_1 \psi_T(w_1)}\ldots e^{-(s_k-s_{k-1}) \psi_T(w_k)} dw_1\ldots dw_k ds_1 \ldots ds_k.
\end{multline}
We would like to take the limit under the integral sign. However, due to the terms $\psi_T$ the use of dominated convergence cannot be justified as simply as in the proof of the stable case. Recall that, by Lemma~\ref{LE5'}, 
\begin{equation}
\int_0^t \int_\mathbb{R} e^{-u\psi_T(z)} dz du
\end{equation}
is bounded uniformly in $T\geq 1$. Now, fix some $K>0$. The integral in~\eqref{loc2} with $\mathbb{R}^k$ replaced by $G_K:=\{(w_1,\ldots,w_k): |w_1|,\ldots,|w_k|\leq K\}$ converges to 
\begin{multline}
k!\bigg(\frac{1}{2\pi}\bigg)^k \bigg(\int_\mathbb{R}\phi(y)dy\bigg)^k \int_{0<s_1 < \ldots < s_k <t}\int_{G_K}e^{ixw_1}e^{-s_1 |w_1|^\beta}\ldots e^{-(s_k-s_{k-1}) |w_k|^\beta}\\
 dw_1\ldots dw_k ds_1 \ldots ds_k,
\end{multline}
by dominated convergence theorem. In view of Lemma~\ref{LE2}, the integral in~\eqref{loc2} with $\mathbb{R}^k$ replaced by $\mathbb{R}^k/G_K$ can be made arbitrarily small for $K$ large enough. By Lemma~\ref{lmoments} in the Appendix
\begin{multline}
k!\bigg(\frac{1}{2\pi}\bigg)^k\int_{0<s_1 < \ldots < s_k <t}\int_{\mathbb{R}^k}e^{ixw_1}e^{-s_1 |w_1|^\beta}\ldots e^{-(s_k-s_{k-1}) |w_k|^\beta}\\
 dw_1\ldots dw_k ds_1 \ldots ds_k=\mathbb{E}L_t(x)^k.
\end{multline}

Very similarly one shows that mixed moments moments of the process 
$(P_t^T(x))_{t\geq 0}$ converge to the mixed moments of the limit process. Thus, we establish the convergence of finite-dimensional distributions. \\

Tightness  under Assumption (\textbf{C}) follows almost immediately. One just has to notice that for $s<t$, a calculation similar to the one in~\eqref{loc2} and Lemma~\ref{LTIGHT1} imply that for $k$ sufficiently large $\mathbb{E}(P_t^T(x) - P_s^T(x))^k$ can be bounded by $(t-s)^\gamma$ for some $\gamma>1$ and then use Kolmogorov's tightness criterion.

\end{proof}

\subsection{Proof of Proposition~\ref{PROP2}}\label{Psec3}

\begin{proof}
Put for $T\geq 1, x \in \mathbb{R}$ and $t>0$
\begin{equation}\label{pot}
\widetilde{P}_t^T(x)=\frac{1}{F_T^{1/2}}\int_0^{Tf(T^{1/\beta})^{-1}t}\phi(T^{1/\beta}x - \eta_s)ds,
\end{equation}
with $F_T$ as in the statement of the proposition. We start by showing that for a fixed $t>0, x \in \mathbb{R}$ and an even positive integer $k$ we have
\begin{multline}\label{ocm}
\lim_{T\rightarrow \infty}\mathbb{E}\big(\widetilde{P}_t^T(x)\big)^k = \frac{k!}{(k/2)!}\Big(\frac{1}{\pi}\int_\mathbb{R}\big|\widehat{\phi}(w)\big|^2\frac{1}{\psi(w)}dw\Big)^{k/2}\mathbb{E}\Big(L_t(x)^{k/2}\Big).
\end{multline}
Similarly as in~\eqref{loc2}, after a change of variables we obtain that $\mathbb{E}\big(\widetilde{P}_t^T(x)\big)^k$ is equal to
\begin{eqnarray}\label{jloc1}
 k!\Big(\frac{1}{2\pi}\Big)^{k}\int_{\mathbb{R}^k}\int_{\mathbb{R}^k}&&\mathbf{1}_{\{0<u_1<D_T^{-1}u_2+u_1<u_3<\ldots<u_{k-1}<D_T^{-1}u_k+u_{k-1}<t\}} \nonumber\\
&& \: \times \widehat{\phi}\big(\frac{w_1}{T^{1/\beta}}-w_2\big)\widehat{\phi}\big(w_2-\frac{w_3}{T^{1/\beta}}\big)\ldots \widehat{\phi}\big(\frac{w_{k-1}}{T^{1/\beta}}-w_k\big)\widehat{\phi}\big(w_k\big)\nonumber\\
&& \: \times e^{iw_1x}e^{-u_1\psi_T(w_1)}e^{-u_2\psi(w_2)}e^{-(u_3-D_T^{-1}u_2-u_1)\psi_T(w_3)}\ldots \times\nonumber\\
&& \: \ldots \times e^{-(u_{k-1}-D_T^{-1}u_{k-2}-u_{k-3})\psi_T(w_{k-1})}e^{-u_k\psi(w_k)} \nonumber\\
&& \:du_1\ldots du_k dw_1 \ldots dw_k.
\end{eqnarray}
For $z,w\in \mathbb{R}$ and $T\geq 1$ let us define
\begin{eqnarray}
a_T(w,z)&=&\widehat{\phi}(w/T^{1\beta}-z)-\widehat{\phi}(-z),\\
b(z)&=&\widehat{\phi}(-z).
\end{eqnarray}
Then~\eqref{jloc1} can be rewritten as
\begin{eqnarray}\label{jloc1'}
 k!\Big(\frac{1}{2\pi}\Big)^{k}\int_{\mathbb{R}^k}\int_{\mathbb{R}^k}&&\mathbf{1}_{\{0<u_1<D_T^{-1}u_2+u_1<u_3<\ldots<u_{k-1}<D_T^{-1}u_k+u_{k-1}<t\}} \nonumber\\
&& \: \times \big(a_T(w_1,w_2)+b(w_2)\big)\big(\overline{a_T(w_2,w_3)}+\overline{b(w_2)}\big)\ldots \nonumber\\
&& \: \times \big(a_T(w_{k-1},w_k)+b(w_k)\big) \big(\overline{b(w_k)}\big)\nonumber\\
&& \: \times e^{iw_1x}e^{-u_1\psi_T(w_1)}e^{-u_2\psi(w_2)}e^{-(u_3-D_T^{-1}u_2-u_1)\psi_T(w_3)}\ldots \times\nonumber\\
&& \: \ldots \times e^{-(u_{k-1}-D_T^{-1}u_{k-2}-u_{k-3})\psi_T(w_{k-1})}e^{-u_k\psi(w_k)} \nonumber\\
&& \:du_1\ldots du_k dw_1 \ldots dw_k.
\end{eqnarray}
We will show that out of all $2^{k-1}$ expressions that we get by multiplying the parentheses with the terms $a_T$ and $b$ in~\eqref{jloc1'}, the only term that does not converge to zero as $T\rightarrow\infty$ is the one in which only $b$'s appear. In fact, we will only prove that the term with
\begin{equation*}
a_T(w_1,w_2)\overline{a_T(w_2,w_3)}\ldots a_T(w_{k-1},w_k)\overline{b_T(w_k)}
\end{equation*}
converges to $0$, the other cases being very similar as the integral with respect to $w_1,\ldots,w_k$ factorizes. Let us denote this term by $M$. Since we assume that $\int_\mathbb{R}|\phi(y)|dy=1$ we see that by Assumption (\textbf{D}) (we can without loss of generality assume that $C=1$ in the formulation of the Assumption (\textbf{D})), 
\begin{eqnarray}\label{jloc1''}
M &\leq& \int_{\mathbb{R}^k}\int_0^t \int_0^{D_Tt}\ldots \int_0^t \int_0^{D_Tt}\nonumber\\
&& \:\big(1\wedge|w_k|^{\kappa}\big)\big(1\wedge|w_{k-1}T^{-1/\beta}|^{2\kappa}\big)\big(1\wedge|w_{k-3}T^{-1/\beta}|^{2\kappa}\big)\nonumber\\
&& \:\times\ldots \times\big(1\wedge|w_3T^{-1/\beta}|^{2\kappa}\big)\big(1\wedge|w_1T^{-1/\beta}|^{\kappa}\big)\nonumber\\
&& \times e^{-u_1\psi_T(w_1)}e^{-u_2\psi(w_2)}e^{-u_3\psi_T(w_3)}\ldots e^{-u_{k-1}\psi_T(w_{k-1})}e^{-u_k\psi(w_k)}\nonumber\\
&& \:du_1\ldots du_k dw_1\ldots dw_k. 
\end{eqnarray}
Now, Lemma~\ref{mloc3} implies that 
\begin{equation}
M\leq c_1(T^{1/\beta}D_T^{-1}+T^{-\kappa/\beta})\times (1+T^{-1/\beta}D_T T^{-\kappa/\beta}),
\end{equation}
for some finite constant $c_1$ independent of $T$. Since $\kappa>(\beta-1)/2$, $M$ converges to $0$ as $T\rightarrow \infty$. The only significant term in~\eqref{jloc1'} is thus given by
\begin{eqnarray}\label{jloc1'''}
 k!\Big(\frac{1}{2\pi}\Big)^{k}\int_{\mathbb{R}^k}\int_{\mathbb{R}^k}&&\mathbf{1}_{\{0<u_1<D_T^{-1}u_2+u_1<u_3<\ldots<u_{k-1}<D_T^{-1}u_k+u_{k-1}<t\}} \nonumber\\
&& \: \times b_T(w_2)\overline{b_T(w_2)}b_T(w_k)\overline{b_T(w_k)}\nonumber\\
&& \: \times e^{iw_1x}e^{-u_1\psi_T(w_1)}e^{-u_2\psi(w_2)}e^{-(u_3-D_T^{-1}u_2-u_1)\psi_T(w_3)}\ldots \times\nonumber\\
&& \: \ldots \times e^{-(u_{k-1}-D_T^{-1}u_{k-2}-u_{k-3})\psi_T(w_{k-1})}e^{-u_k\psi(w_k)} \nonumber\\
&& \:du_1\ldots du_k dw_1 \ldots dw_k,
\end{eqnarray}
which converges, by dominated convergence theorem, to the right-hand side of~\eqref{ocm} (see Lemma~\ref{lmoments}). \\

Very similarly one shows that for all odd positive integers $k$ the respective moments converge to $0$ and that for any $t_1\leq \ldots t_k$ and $x\in \mathbb{R}$
\begin{equation}
\mathbb{E}\big(\widetilde{P}_{t_1}^T(x)\ldots \widetilde{P}_{t_k}^T(x)\big) \rightarrow c(\phi)^k\mathbb{E}(W_{L_{t_1}(x)}^k).
\end{equation}

Showing that, provided the Assumption (\textbf{C}) holds, the sequence of processes on the left-hand side of~\eqref{ocm} is tight amounts to repeating the arguments used in the proof of Lemma~\ref{LTIGHT1}. One must just notice that given our assumptions there is a constant $c_2$ independent of $T$ such that for $k$ even
\begin{multline}
\mathbb{E}\big(\widetilde{P}_t^T(x)-\widetilde{P}_s^T(x)\big)^k \leq \\
\int_{\mathbb{R}^{k/2}}\int_{[0,t-s]^{k/2}}e^{-u_1\psi_T(w_1)}e^{-u_3\psi_T(w_3)}\ldots e^{-u_{k-1}\psi_T(w_{k-1})}\\
du_1du_3\ldots du_{k-1} dw_1dw_3\ldots dw_{k-1}.
\end{multline}
This and~\eqref{nloc1} imply that 
\begin{equation}
\mathbb{E}\big(\widetilde{P}_t^T(x)-\widetilde{P}_s^T(x)\big)^k \leq c_3 (t-s)^{k\delta/2}
\end{equation}
for some finite constant $c_3$ independent of $T$. Taking $k$ large enough we may apply the Kolmogorov's tightnes criterion (see Theorem 12.3 in~\cite{BILL}) and infer that the sequence of processes $(P^T(x))$ is tight in $\mathcal{C}[0,\infty)$.
\end{proof}

\section{Proof of Theorem~\ref{THM1}}\label{sec_fol}

\begin{proof}
Let $a_k \in \mathbb{R}, t_k\geq 0$, $k=1,\ldots,m$, put $\bar{G}^T = \sum_k^m a_kG_{t_k}^T$ and $\bar{P}^T(x) = \sum_k^m a_kP_{t_k}^T(x)$ for $t\geq 0,T\geq 1$ and $x\in \mathbb{R}$, with $P^T$ and $G^T$ defined by~\eqref{P_def} and~\eqref{G1}, respectively. The characteristic function of $\bar{G}^T$, after a change of variables $z:=T^{1/\alpha \beta}z$ and $x:=-T^{1/\beta}C_T^{-1}x$, can be written as 
\begin{multline}\label{GCHAR}
\mathbb{E} exp(i \bar{G}^T)= \\
exp\Bigg(\int_{\mathbb{R}^2}\mathbb{E}\bigg(e^{i z \bar{P}^T(x)}-1\bigg)\mathbf{1}_{\{|z|>T^{-1/\alpha \beta}\epsilon\}}
|z|^{-\alpha-1}\frac{L(T^{1/\alpha \beta}z)}{L(T^{1/\alpha \beta})}dzdx\Bigg).
\end{multline}
By the symmetry of function $L$, ~\eqref{GCHAR} can be rewritten as
\begin{eqnarray}
\mathbb{E} exp(i\theta \bar{G}^T)&=&exp\Big(\int_{\mathbb{R}^2}\mathbb{E}\big(e^{i\theta z \bar{P}^T(x)}-\mathbf{1}_{\{|z|\leq 1\}}i\theta z\bar{P}^T(x) - 1\big)\nonumber\\
&& \: \times\mathbf{1}_{\{|z|>T^{-1/\alpha \beta}\epsilon\}}
|z|^{-\alpha-1}\frac{L(T^{1/\alpha \beta}z)}{L(T^{1/\alpha \beta})}dzdx\Big).\label{GCHARS}
\end{eqnarray}
The quantity in~\eqref{GCHAR}, by Proposition~\ref{PROP1}, converges pointwise as $T\rightarrow \infty$ to 
\begin{equation}
\mathbb{E}exp\left(i\theta\int_{\mathbb{R}}\phi(y)dy\sum_{k=1}^m a_m X_{t_k}\right),
\end{equation}
with $X$ being the process given by~\eqref{G1}, so it remains to justify that we can go with the limit under the integral sign. This requires a number of observations, which, for greater clarity, are given in a lemma below.

\begin{lem}\label{GLEM}
Let $P^T$ be as in~\eqref{P_def} and assume that the conditions of Theorem~\ref{THM1} are satisfied. Then the following claims are true.
\begin{enumerate}
\item[(i)]
The functions 
\begin{equation}
x \mapsto \mathbb{E}P_t^T(x),
\end{equation}
and
\begin{equation}
x \mapsto \mathbb{E}P_t^T(x)^2,
\end{equation}
are bounded uniformly in $T\geq 1$.
\item[(ii)]
Both $\int_\mathbb{R}\mathbb{E}|P_t^T(x)|dx$ and $\int_\mathbb{R}\mathbb{E}|P_t^T(x)|^2dx$ are bounded uniformly in $T\geq 1$.
\item[(iii)]
For any $\delta>0$ and there exist $K>0$ and $T_0\geq 1$ such that both
\begin{equation}
\int_{|x|>K} \mathbb{E}|P_t^T(x)|dx<\delta
\end{equation}
and
\begin{equation}
\int_{|x|>K} \mathbb{E}|P_t^T(x)|^2dx<\delta,
\end{equation}
for all $T\geq T_0$.
\item[(iv)]
\begin{equation}
\lim_{T\rightarrow \infty}\int_{\mathbb{R}}\int_{|z|\leq T^{-1/\alpha \beta}\epsilon}\mathbb{E}\left(|z|^2|P_t^T(x)|^2\right)|z|^{-\alpha-1}\frac{L(T^{1/\alpha \beta}z)}{L(T^{1/\alpha \beta})}dzdx = 0.
\end{equation}
\item[(v)]
For any $r\in (0,1)$ there exists $T_0\geq 1$ and constant $C$, depending only on $\alpha$, such that 
\begin{equation}
\int_{|z|\leq r}|z|^{1-\alpha}\frac{L(T^{1/\alpha \beta}z)}{L(T^{1/\alpha \beta})}dz \leq Cr^{2-\alpha}.
\end{equation}

\end{enumerate}
\end{lem}

\begin{proof}[Proof of Lemma~\ref{GLEM}]
Changing variables of variables and using Plancherel and Fubini's theorems, for $x\in \mathbb{R},t\geq 0, T>0$ and $\phi \in L^1(\mathbb{R})$ we have
\begin{equation}
\mathbb{E}P_t^T(x)=\frac{1}{2\pi}\int_0^t\int_\mathbb{R}\widehat{\phi}\bigg(\frac{w}{T^{1/\beta}}\bigg)e^{ixw}e^{-u\psi_T(w)}dw,
\end{equation}
where $\psi_T$ is as in Lemma~\ref{clim}. Hence 
\begin{equation}
\mathbb{E}|P_t^T(x)| \leq \frac{1}{2\pi}\norm{\phi}_1\int_0^t\int_\mathbb{R}e^{-u\psi_T(w)}dwdu,
\end{equation}
which is bounded uniformly in $T \geq 1$ by Lemma~\ref{LE5'}. Using similar techniques one can write
\begin{eqnarray*}
\mathbb{E}P_t^T(x)^2 &=& \frac{2}{(2\pi)^2}\int_0^t\int_{u_1}^t \int_{\mathbb{R}^2}e^{ixw_1}\widehat{\phi}\bigg(\frac{w_1-w_2}{T^{1/\beta}}\bigg)\widehat{\phi}\bigg(\frac{w_2}{T^{1/\beta}}\bigg)\\
&& \:e^{-(u_2-u_1)\psi_T(w_2)}e^{-(u_1)\psi_T(w_1)}dw_1dw_2du_1du_2\\
&& \: \leq \frac{1}{2\pi^2}\Big(\int_0^t \int_\mathbb{R}e^{-u\psi_T(w)}dw du\Big)^2
\end{eqnarray*}
and argue similarly. This proves (ii).
\\

Obviously for any $t\geq 0, T\geq 1$, by replacing $\phi$ with its absolute value we get
\begin{equation*}
\Big|\int_\mathbb{R}\mathbb{E}P_t^T(x)dx\Big|\leq \int_\mathbb{R}\mathbb{E}\big|P_t^T(x)\big|dx \leq\norm{\phi}_1 t.
\end{equation*}
As for the second part of (ii), using the same techniques as in the proof of (i), we may write
\begin{multline}
\int_\mathbb{R} \mathbb{E}P_t^T(x)^2 dx = \\
 \int_\mathbb{R} \frac{1}{\pi}\int_0^t \int_{u_1}^t \int_\mathbb{R}\phi(-x)e^{ixw}\widehat{\phi}(w)T^{1/\beta}e^{-D_T(u_2-u_1)\psi(w)}dwdu_1du_2dx.
\end{multline}
The above can be bounded by 
\begin{equation}\label{oloc5}
\frac{1}{\pi}\norm{\phi}_1\int_0^t \int_{u_1}^t \int_\mathbb{R}\bigg|\widehat{\phi}\bigg(\frac{w}{T^{1/\beta}}\bigg)\bigg|e^{-(u_2-u_1)\psi_T(w)}dwdu_1du_2,
\end{equation}
which in turn is no bigger than
\begin{equation}
\frac{1}{\pi}\norm{\phi}_1^2\int_0^t \int_{u_1}^t \int_\mathbb{R}e^{-(u_2-u_1)\psi_T(w)}dwdu_1du_2.
\end{equation}
By Lemma~\ref{LE5'}, the last expression is bounded uniformly in $T\geq 1$. This proves (ii).
\\

Let us now turn to showing (iii). In order to escape notational complexity we will only consider the integrals over $\{x\in \mathbb{R}:x>K\}$. For $\{x\in \mathbb{R}:x< K\}$ it  is then enough to use the symmetry of $\eta$ and take $\widetilde{\phi}(x)=\phi(-x)$. First notice that after some simple manipulations we get by Fubini theorem
\begin{eqnarray*}
\int_K^\infty \mathbb{E}P_t^T(x)dx &=& \mathbb{E}\Big(\int_0^t\int_\mathbb{R}\mathbf{1}_{\{y<\eta_{D_Tu}-KT^{1/\beta}\}}\phi(y)dydu\Big)\\
&=& \int_0^t\int_\mathbb{R}\mathbb{P}(yT^{-1/\beta}<T^{-1/\beta}\eta_{D_Tu}-K)\phi(y)dy,
\end{eqnarray*}
which converges as $T\rightarrow \infty$, by dominated convergence theorem and~\eqref{LSCAL}, to

\begin{equation}
\int_{\mathbb{R}}\phi(y)dy \int_0^t\mathbb{P}(\xi_u>K)du = \int_{\mathbb{R}}\phi(y)dy \int_K^\infty \mathbb{E}L_t(x)dx.
\end{equation}
By choosing $K$ large enough to begin with and using Lemma~\ref{lc} from the Appendix we see that the first part of (iii) is true. Regarding its second part, write (again after some initial manipulations and using Fubini's theorem)
\begin{eqnarray}
\int_K^\infty \mathbb{E}P_t^T(x)^2dx &=& 2 \int_0^t \int_{u_1}^t \int_\mathbb{R} \mathbb{E}\Big(\mathbf{1}_{\{x>KT^{1/\beta} - \eta_{D_Tu_1}\}}\phi(-x) \\
&& \: \times T^{1/\beta}\phi(\eta_{D_Tu_2}-\eta_{D_Tu_1}-x)\Big)dxdu_1du_2.
\end{eqnarray}
Since $\eta$ is a L\'{e}vy process the above equals %(here $\widetilde{\eta}$ denotes a L\'{e}vy process independent of $\eta$, with the same distribution)
\begin{eqnarray} 
2 \int_0^t \int_{u_1}^t \int_\mathbb{R}&& \mathbb{P}\big(x>KT^{1/\beta} - \eta_{D_Tu_1}\big)\phi(-x)\\
&& \:\times T^{1/\beta}\mathbb{E}\big(\phi(\eta_{D_T(u_2-u_1)}-x)\big)dxdu_1du_2\nonumber\\
&=& 2 \int_0^t \int_{u_1}^t \int_\mathbb{R}\mathbb{P}\big(xT^{-1/\beta}>K - T^{1/\beta}\eta_{D_Tu_1}\big)\phi(-x) \label{gloc3}\\
&& \times \frac{1}{2\pi} \int_\mathbb{R}\widehat{\phi}\bigg(\frac{w}{T^{1/\beta}}\bigg)e^{-ixwT^{-1/\beta}}e^{-(u_2-u_1)\psi_T(w)}dw dxdu_1du_2.\nonumber
\end{eqnarray}
The integrand in~\eqref{gloc3} the above can be bounded by $|\phi(x)|\norm{\phi}_1e^{-(u_2-u_1)\psi_T(w)}$ which is integrable by Lemma~\ref{LE5'}. Therefore, 
\begin{equation}
\limsup_{T\rightarrow \infty} \int_K^\infty \mathbb{E}P_t^T(x)^2dx \leq c_1 \left(\int_\mathbb{R}\phi(y)dy\right)^2\int_K^\infty \mathbb{E}L_t(x)^2dx,
\end{equation}
for some constant $c_1$ independent of $T$ and $K$. In view of Lemma~\ref{lc}, this ends the proof of (iii). The proofs of (iv) and (v) are relatively straightforward consequences of Theorem 10.5.6 in~\cite{SPLRD} and we skip them. We also skip the proof of (vi)
\end{proof}

Given Lemma~\ref{GLEM} we will show that for any $K>0$ and $r \in (0,1)$ the integrand in
\begin{eqnarray}\label{gloc1}
\int_{|x|\leq K}\int_{|z|\geq r} &&\mathbb{E}\bigg(e^{i\theta z \widetilde{P}^T(x)}-\mathbf{1}_{\{|z|\leq 1\}}i\theta z\widetilde{P}^T(x) - 1\bigg)\nonumber\\
&& \:\times|z|^{-\alpha-1}\frac{L(T^{1/\alpha \beta}z)}{L(T^{1/\alpha \beta})}dzdx
\end{eqnarray}
can be bounded (uniformly in $T\geq 1$) by an integrable function and that, by dominated convergence, this is enough to prove the convergence of finite-dimensional distributions. Indeed (using inequalities $|e^{iw}-1|\leq |w|$ and $|e^{iw}-w-1|\leq \frac{1}{2}|w|^2$ for $w\in \mathbb{R}$), 
\begin{eqnarray*}
&&\int_{|x|> K}\int_{\mathbb{R}} \mathbb{E}\bigg(e^{i\theta z \widetilde{P}^T(x)}-\mathbf{1}_{\{|z|\leq 1\}}i\theta z\widetilde{P}^T(x) - 1\bigg)\\
&& \:\times|z|^{-\alpha-1}\frac{L(T^{1/\alpha \beta}z)}{L(T^{1/\alpha \beta})}dzdx\\
&\leq& \int_{|z|\leq 1}|z|^{1-\alpha}\frac{L(T^{1/\alpha \beta}z)}{L(T^{1/\alpha \beta})}dz\int_{|x|>K}\mathbb{E}|\widetilde{P}^T(x)|^2dx\\
&& \:+ \int_{|z|> 1} |z|^{-\alpha}dz \int_{|x|>K}\mathbb{E}|\widetilde{P}^T(x)|dx
\end{eqnarray*}
and
\begin{eqnarray*}
&&\int_{\mathbb{R}}\int_{|z|\leq r}\mathbb{E}\bigg(e^{i\theta z \widetilde{P}^T(x)}-\mathbf{1}_{\{|z|\leq 1\}}i\theta z\widetilde{P}^T(x) - 1\bigg)\\
&& \:\times|z|^{-\alpha-1}\frac{L(T^{1/\alpha \beta}z)}{L(T^{1/\alpha \beta})}dzdx\\
&\leq& \int_{\mathbb{R}}\mathbb{E}|\widetilde{P}_t^T(x)|^2dx \int_{|z|\leq r}|z|^{1-\alpha}\frac{L(T^{1/\alpha \beta}z)}{L(T^{1/\alpha \beta})}dz,
\end{eqnarray*}
which in view of Lemma~\ref{GLEM} can be made arbitrarily small for all $T$ sufficiently large by first choosing $K$ large enough and $r$ small enough. 

By Proposition 10.5.5 and Corollary 10.5.8 in~\cite{SPLRD}, for $r\in (0,1)$ fixed there exists $T_0\geq 1$ such that 
\begin{equation}
\left|\frac{L(T^{1/\alpha \beta}z)}{L(T^{1/\alpha \beta})}-1\right|\leq 1
\end{equation}
for all $z \in (r,1]$ and $T\geq T_0$ and for any $\delta>0$ there exists $T_1\geq 1$ such that for all $z\geq 1$ and $T\geq T_1$ we have 
\begin{equation}
\left|\frac{L(T^{1/\alpha \beta}z)}{L(T^{1/\alpha \beta})}\right|\leq (1+\delta)|z|^\delta.
\end{equation}
This impies that for $|z|>r$ and all $T$ large enough the function 
\begin{multline}
(x,y)\mapsto \mathbb{E}\Big(e^{i\theta z P^T(x)}-\mathbf{1}_{\{|z|\leq 1\}}i\theta zP^T(x) - 1\Big)|z|^{-\alpha-1}\frac{L(T^{1/\alpha \beta}z)}{L(T^{1/\alpha \beta})}
\end{multline}
can be bounded by the function
\begin{multline}
(x,y)\mapsto 2\mathbb{E}|P_t^T(x)|^2\mathbf{1}_{\{|z|\in (r,1]\}}|\theta|^2|z|^{1-\alpha}\\
 + \mathbb{E}|P_t^T(x)|\mathbf{1}_{\{|z|>1\}}|z|^{-\alpha}(1+\delta)|z|^{\delta}.
\end{multline}
Choosing $\delta$ small enough and again using Lemma~\ref{GLEM} we see that the above can be bounded by an integrable function, uniformly for all $T$ large enough.
\end{proof}

\begin{lem}
Assume that $\psi$ satisfies the conditions of Lemma~\ref{LTIGHT1}. Then the family of processes $\{(G^T_t)_{t\geq 0}:T\geq 1\}$ defined by~\eqref{GF1} is tight in $\mathcal{C}[0,\tau)$ for any $\tau>0$.
\end{lem}

\begin{proof}
For any $K>0,T\geq 1$ put $K_T:=KT^{1/\alpha \beta}$ and let $G_t^T = G_t^{T,1}+G_t^{T,2}$ for any $t\geq 0$, with
\begin{equation}
G_t^{T,1}:=\frac{1}{N_T}\sum_j z_j\mathbf{1}_{\{K_T>|z_j|>\epsilon\}}\int_0^{Tt}\phi(C_Tx^j+\xi^j_u)du,
\end{equation}
and
\begin{equation}
G_t^{T,2}:=\frac{1}{N_T}\sum_j z_j\mathbf{1}_{\{|z_j|\geq K_T\}}\int_0^{Tt}\phi(C_Tx^j+\xi^j_u)du.
\end{equation}
We are going to show that the family of processes $(G^{T,1}_t)_{t\geq 0}$ is tight $\mathcal{C}[0,\tau]$ for any $\tau>0$ and that for any $\delta>0$
\begin{equation}\label{oloc6}
\lim_{K\rightarrow \infty} \limsup_{T\rightarrow \infty}\mathbb{P}(\sup_{t\in [0,\tau]}|G_t^{T,2}|>\delta)=0,
\end{equation}
which suffices to establish tightness.
\\

We now proceed to establish tightness for the family $(G_t^{T,1})_{t\geq 0}$. Notice that
\begin{multline}\label{oloc4}
\mathbb{E}\big(G_t^{T,1}-G_s^{T,1}\big)^2 \leq c_1 \int_\mathbb{R}\mathbb{E}\big(P_{t-s}^T(x)^2\big)dx\int_{|z|\leq K}|z|^{1-\alpha}\frac{L(T^{1/\alpha \beta}z)}{L(T^{1/\alpha \beta})}dz,
\end{multline}
for some finite constant $c_1$. After a change of variables $z:= zT^{1/\alpha \beta}$ and an application of Theorem 10.5.6 in~\cite{SPLRD}, we conclude that for all $T$ large enough, the integral over $\{|z|\leq K\}$ in~\eqref{oloc4} is bounded by $c_2K^{2-\alpha}$ for some finite constant $c_2$ depending only on $\alpha$. Furthermore, by~\eqref{oloc5}, the integral $\int_\mathbb{R}\mathbb{E}\big(P_{t-s}^T(x)^2\big)dx$ can be bounded by
\begin{eqnarray}
&&\frac{1}{\pi}\norm{\phi}_1\int_s^t \int_{u_1}^t \int_\mathbb{R}\bigg|\widehat{\phi}\bigg(\frac{w}{T^{1/\beta}}\bigg)\bigg|e^{-(u_2-u_1)\psi_T(w)}dwdu_1du_2\\
&&\leq \frac{1}{\pi}\norm{\phi}_1^2(t-s) \int_0^{t-s} \int_\mathbb{R}e^{-u_2\psi_T(w)}dwdu_2\label{yloc1}\\
\end{eqnarray}
Using Lemma~\ref{LTIGHT1}, we see that~\eqref{yloc1} is bounded by
\begin{equation*}
c_3(t-s)^{1+\delta}
\end{equation*}\
for a constant $c_3$ independent of $s,t$ and $T$ and some $\delta>0$.An application of Theorem 12.3 in~\cite{BILL} shows that the family $(G_t^{T,1})_{t\geq 0}$ is tight in $\mathcal{C}[0,\tau]$ for any $\tau>0$.
\\

Proceeding further, notice that for any $\delta, \tau>0$ (after a change of variables)
\begin{multline}
\mathbb{P}(\sup_{t\in [0,\tau]}|G_t^{T,2}|>\delta)\\
\leq \frac{1}{\delta}\int_{\mathbb{R}}\int_{|z|\geq K}\mathbb{E}\sup_{t\in[0,\tau]}\big|P_t^T(x)||z|^{-\alpha}\frac{L(T^{1/\alpha \beta}z)}{L(T^{1/\alpha \beta})}dzdx,
\end{multline}
which by part (iii) of Lemma~\ref{GLEM} (with $\phi$ replaced by its absolute value) and Corollary 10.5.8 in~\cite{SPLRD} can, for all $T$ large enough, be bounded by
\begin{equation}
c_4 \int_{|z|\geq K}|z|^{-\alpha}|z|^{\delta}dz,
\end{equation}
with $c_4$ being a constant independent of $T$ and $\delta>0$ can be arbitrarily small. This establishes~\eqref{oloc6} and finishes the proof of the lemma.

\end{proof}

\section{Proofs for Section~\ref{RS2}}\label{sec_sol}

\begin{comment}
When $\int_\mathbb{R}\phi(y)dy=0$ in the scheme described in the beginning of the paper, things get a little bit more involved and the the behaviour of $\phi(y)$ as $|y|\rightarrow \infty$ determines what kind of limiting process of~\eqref{FC} can we expect. First we investigate the case when $\phi$ has relatively light tails. In this case Proposition~\ref{PROP2} comes into play and enables us to provide an alternative representation of the processes first investigated in~\cite{SAM1} and~\cite{SAM2}. We then briefly describe what happens for $\phi$ with sufficiently heavy tails and prove Theorem~\ref{THM3}.
\end{comment}

%One can also consider the case with $\gamma_1=\gamma_2$, in which case the limit process is either $G$ or $\widetilde{G}$ depending on the behavior of $\frac{g_1(T)}{g_2(T)}$ as $T\rightarrow \infty$.

\subsection{Proof of Theorem~\ref{THM2}}
Let us now consider the case in which $\phi \in L^1(\mathbb{R})$, $\int_\mathbb{R}\phi(y)=0$ and $\phi$ vanishes relatively quickly as $|y|\rightarrow \infty$. This would enable us to use Proposition~\ref{PROP2} in order to find a limit of the functional in~\eqref{FC}.

Put 
\begin{equation}\label{R_def}
R_t^T(x):=\int_\mathbb{R}\phi(y)L_t(x+T^{-1/\beta}y)dy,
\end{equation}
for $x \in \mathbb{R}$, $t\geq 0$ and $T>0$. Before we prove Theorem~\ref{THM2} we will need a couple of auxiliary facts which are given in the Lemmas~\ref{sloc1} and~\ref{sloc2} below.

\begin{lem}\label{sloc1}
Assume that the conditions of Theorem~\ref{THM2} are satisfied. Then for every $T\geq 1$ and $\phi \in L^1(\mathbb{R})$ we have
\begin{equation}\label{w1}
I_1^T:=T^{\frac{\beta-1}{2\beta}}\int_\mathbb{R}\mathbb{E}|R_t^T(x)|dx < \infty,
\end{equation}
and
\begin{equation}\label{w2}
I_2^T:=T^{\frac{\beta-1}{\beta}}\int_\mathbb{R}\mathbb{E}|R_t^T(x)|^2 dx < \infty.
\end{equation}
If, in addition, we assume that $\int_\mathbb{R}|\phi(y)||y|^{\frac{\beta-1}{2}}dy <\infty$, then 
\begin{eqnarray}
\sup_{T\geq 1}I_1^T&<&\infty,\label{ur1}\\
\sup_{T\geq 1}I_2^T&<&\infty.\label{ur2}
\end{eqnarray}
\end{lem}
\begin{proof}

It is not hard to see, using Lemma~\ref{gbound} from the Appendix, that for any $z,x \in \mathbb{R}$
\begin{equation}\label{sb}
\mathbb{E}\big(L_t(x+z)-L_t(x)\big)^2 \leq 2 \big(\mathbb{E}L_t(x+z)+\mathbb{E}L_t(x)\big)\big(c_1(t,\beta)\wedge c_2(t,\beta)|z|^{\beta -1}\big),
\end{equation}
for some constants $c_1,c_2$ depending only on $t$ and $\beta$. By H\"{o}lder inequality, Lemma~\ref{gbound} and the fact that $\int_\mathbb{R}\phi(y)dy=0$ we get
\begin{eqnarray*}
I_1^T &\leq& \int_\mathbb{R} \int_\mathbb{R} T^{\frac{\beta-1}{2\beta}}|\phi(y)|\Bigg(2 \Big(\mathbb{E}L_t(x+yT^{-1/\beta})+\mathbb{E}L_t(x)\Big)\\
&& \:(c_1\wedge c_2|yT^{1/\beta}|^{\beta -1})\Bigg)^{\frac{1}{2}}dydx\\
&=& \int_\mathbb{R} \int_\mathbb{R} |\phi(y)|\Big(2 \big(\mathbb{E}L_t(x+yT^{-1/\beta})+\mathbb{E}L_t(x)\big)\\
&& \:(T^{\frac{\beta-1}{\beta}}c_1\wedge c_2|y|^{\beta -1})\Big)^{\frac{1}{2}}dydx\\
&\leq&\int_\mathbb{R} \int_\mathbb{R} |\phi(y)|\Big(2 \big(c_3\wedge c_4|x+yT^{-1/\beta}|^{-\beta-1}+c_3\wedge c_4|x|^{-\beta-1}\big)\\
&& \:(T^{\frac{\beta-1}{\beta}}c_1\wedge c_2|y|^{\beta -1})\Big)^{\frac{1}{2}}dxdy,\\
\end{eqnarray*}
where $c_3,c_4$ depend only on $t$ and $\beta$. Thus, $I_1^T$ is finite since $\beta+1>2$. Note that it is bounded uniformly in $T\geq 1$ for $\phi \in L^1(\mathbb{R})$ and $\int_\mathbb{R}|\phi(y)||y|^{\frac{\beta-1}{2}}dy <\infty$. \\

The proof of~\eqref{w2} is essentially the same once we use H\"{o}lder inequality. We will focus only on showing that  $\sup_{T\geq 1}I_2^T<\infty$ for $\phi$ satisfying $\int_\mathbb{R}|\phi(y)||y|^{(\beta-1)/2}dy < \infty$. Using~\eqref{sb}, we get
\begin{eqnarray*}\label{w3}
I_2^T &\leq& T^{\frac{\beta-1}{\beta}}\int_{\mathbb{R}}\Bigg(\int_\mathbb{R}|\phi(y)|\Big(\mathbb{E}\big(L_t(x+yT^{-1/\beta})-L_t(x)\big)^2\Big)^{1/2}dy\Bigg)^2dx.\\
&\leq& \int_{\mathbb{R}}\Bigg(\int_\mathbb{R}|\phi(y)|\Big(2\big(\mathbb{E}L_t(x+yT^{-1/\beta})+\mathbb{E}L_t(x)\big)c_2|y|^{\beta-1}\Big)^{1/2}dy\Bigg)^2dx
\end{eqnarray*}
Seeing that $\int_\mathbb{R}|\phi(y)|\Big(2\mathbb{E}L_t(x+yT^{-1/\beta})+2\mathbb{E}L_t(x)\Big)^{1/2}(c_2|y|^{\beta-1})^{1/2}dy$ can be bounded by $c_5\int_\mathbb{R}|\phi(y)||y|^{(\beta-1)/2}dy$ for some constant $c_5$ depending only on $\beta$ and $t$, we see (using the same arguments as in the proof of~\eqref{ur1}) that~\eqref{ur2} holds.
\end{proof}

\begin{comment}
Using Lemma~\ref{unif} we can show that~\eqref{xbound1} holds uniformly in $T \geq 1$ as long as 
\begin{equation}\label{con1}
\int_\mathbb{R}|\phi(y)||y|^{\frac{\beta-1}{2}}dy <\infty,
\end{equation}
and we may proceed to the main proof of this section.
\end{comment}

\begin{lem}\label{sloc2}
Let $t\geq 0$ and assume that the integrable function $\phi$ satisfies the assumptions in the statement of Theorem~\ref{THM2}. Then, for any $\delta>0$ there exists $K_0>0$ and $T_0=T_0(K_0)$ such that for all $T\geq T_0, K\geq K_0$ we have
\begin{equation}\label{RTY}
\int_{\{|x|>K\}}T^{\frac{\beta-1}{2\beta}}\mathbb{E}\big|R_t^T(x)\big|dx < \delta.
\end{equation}
\end{lem}

\begin{proof}
Choose $K_0$ so that
\begin{equation}
\int_{\{|x|>K_0\}} c_2(t,\beta)\int_\mathbb{R}|\phi(y)||y|^{\frac{\beta-1}{2}}dy dx < \frac{\delta}{2},
\end{equation}
where $c_2(t,\beta)$ is the same as in~\eqref{sb}. Using H\"{o}lder inequality and inequality~\eqref{sb} we have
\begin{multline} 
\int_{\{|x|>K\}}T^{\frac{\beta-1}{2\beta}}\mathbb{E}\big|R_t^T(x)\big|dx \leq \\
\leq \int_{\{|x|>K\}} c_2(t,\beta)\int_\mathbb{R}\big(\mathbb{E}L_t(x)+\mathbb{E}L_t(x+T^{-1/\beta})\big)^{\frac{1}{2}})|\phi(y)||y|^{\frac{\beta-1}{2}}dy dx.
\end{multline}
The above can be bounded by $A+B$, with
\begin{eqnarray*}
A&\leq&  \int_{\{|x|>K_0\}} 2c_2(t,\beta)\int_\mathbb{R}\big(\mathbb{E}L_t(x)\big)^{\frac{1}{2}}|\phi(y)||y|^{\frac{\beta-1}{2}}dy dx,\\
B&\leq&  \int_{\{|x|>K_0\}} 2c_2(t,\beta)\int_\mathbb{R}\big(\mathbb{E}L_t(x+T^{-1/\beta})\big)^{\frac{1}{2}}|\phi(y)||y|^{\frac{\beta-1}{2}}dy dx.
\end{eqnarray*}
$B$ can be rewritten as 
\begin{equation}
2c_2(t,\beta)\int_\mathbb{R}\int_\mathbb{R}\mathbf{1}_{\{|x-T^{-1/\beta}|>K\}}|\phi(y)||y|^{\frac{\beta-1}{2}}\big(\mathbb{E}L_t(x)\big)^{\frac{1}{2}}dydx,
\end{equation}
which, by dominated convergence, converges to 
\begin{equation}
2c_2(t,\beta)\int_{\{|x|>K_0\}}\int_\mathbb{R}|\phi(y)||y|^{\frac{\beta-1}{2}}\big(\mathbb{E}L_t(x)\big)^{\frac{1}{2}}dydx,
\end{equation}
as $T\rightarrow \infty$. Choosing $T_0$ sufficiently large, we get the required inequality for all $K\geq K_0$.
\end{proof}

\begin{rem}
Using H\"{o}lder inequality one can easily show that the statement in Lemma~\ref{sloc2} remains true if in~\eqref{RTY} we replace 
\begin{equation*}
T^{\frac{\beta-1}{2\beta}}\mathbb{E}\big|R_t^T(x)\big|
\end{equation*}
by
\begin{equation*}
T^{\frac{\beta-1}{\beta}}\mathbb{E}R_t^T(x)^2
\end{equation*}
and the proof is then an easy consequence of Lemma~\ref{sloc2}.
\end{rem}

\begin{proof}[Proof of Theorem~\ref{THM2}]
\begin{comment}
Recall, that we consider the family of processes given by
\begin{equation}
\widetilde{G}_t^T=\frac{1}{\widetilde{N}_T}\sum_j \mathbf{1}_{\{|z_j|>\epsilon\}}\int_0^T\phi(x^j+\eta_u^j)du, \quad t\geq0,
\end{equation}
where $T\geq 1$, $\epsilon$ is an arbitrary positive constant and the normalization is given by
\begin{eqnarray}
\widetilde{N}_T = T^{\frac{\beta-1}{2\beta}+\frac{1}{\alpha \beta}}.
\end{eqnarray}
\end{comment}
Let $a_1,\ldots,a_m \in \mathbb{R}$ and $t_1,\ldots,t_m\geq 0$ for some $m\geq 1$. Then the characteristic function of $\sum_{j=1}^m a_j \widetilde{G}_{t_j}^T$ is given (after the usual change of variables) by
\begin{eqnarray}\label{sloc3}
&exp&\Bigg(\int_{\mathbb{R}^2}\mathbb{E}\left(e^{iT^{\frac{\beta-1}{2\beta}}z\sum_{j=1}^m a_j R_{t_j}^T(x)}-i\mathbf{1}_{\{|z|\leq 1\}}T^{\frac{\beta-1}{2\beta}}z \sum_{j=1}^m a_j R_{t_j}^T(x)-1\right)\nonumber\\
&& \:\mathbf{1}_{\{|z|>T^{-1/(\alpha \beta)}\}}|z|^{-1-\alpha}dzdx\Bigg),\quad \theta \in \mathbb{R}.
\end{eqnarray}
Notice first that by Lemma~\ref{sloc1}, in the limit as $T\rightarrow \infty$, we can forget about the term $\mathbf{1}_{\{|z|>T^{-1/(\alpha \beta)}\}}$. Without it the expression in~\eqref{sloc3} equals, by Lemma~\ref{trick} in the Appendix,
\begin{equation}\label{sloc4}
exp\left(-c_1(\alpha)\int_\mathbb{R}\mathbb{E}\Big|T^{\frac{\beta-1}{2\beta}}\sum_{j=1}^m a_j R_{t_j}^T(x)\Big|^\alpha dx \right),
\end{equation}
for some finite constant $c_1$ depending only on $\alpha$. Since for any random variable $Z$ and $\alpha \in (1,2)$, $\mathbb{E}|Z|^\alpha \leq \mathbb{E}|Z|+\mathbb{E}|Z|^2$, using Lemma~\ref{sloc2} with $K_0$ large enough we can make the integral
\begin{equation}
\int_{\{|x|>K_0\}}\mathbb{E}\Big|T^{\frac{\beta-1}{2\beta}}\sum_{j=1}^m a_j R_{t_j}^T(x)\Big|^\alpha dx
\end{equation}
arbitrarily small, uniformly for all $T$ large enough. By Proposition~\ref{PROP2} (or rather its proof which establishes convergence of all moments of the Process $(T^{\frac{\beta-1}{2\beta}}R_t^T(x))_{t\geq 0}$), the quantity
\begin{equation}
\mathbb{E}\Big|T^{\frac{\beta-1}{2\beta}}\sum_{j=1}^m a_j R_{t_j}^T(x)\Big|^\alpha
\end{equation}
converges, as $T\rightarrow \infty$ to
\begin{equation}
\mathbb{E}\Big|c_2(\phi,\beta)\sum_{j=1}^m a_j W_{t_j}(x)\Big|^\alpha
\end{equation}
for any $x \in \mathbb{R}$, with $c_2(\phi,\beta)=\sqrt{\int_\mathbb{R}|\widehat{\phi}(y)|^2|y|^{-\beta}dy}$. By dominated convergence theorem we conclude that~\eqref{sloc4} converges to
\begin{equation}\label{sloc5}
exp\left(-c_1(\alpha)\int_\mathbb{R}\mathbb{E}\Big|c_2(\phi,\beta)\sum_{j=1}^m a_j W_{t_j}(x)\Big|^\alpha dx \right).
\end{equation}
This finishes the proof of the convergence of finite-dimensional distributions.
\end{proof}

\subsection{Regular variation, heavy tails and the Proof of Theorem~\ref{THM3}}\label{FATtails}
First, let us concentrate on a very concrete choice of $\phi$ to show what happens when $\phi$ vanishes relatively slowly at infinity. We will then extend our discussion to the case of functions regularly varying at infinity.

\subsubsection{A simple example}

Suppose that
\begin{equation}
\phi(y):=|y|^{-\gamma}\mathbf{1}_{\{y\geq 1\}}-|y|^{-\gamma}\mathbf{1}_{\{y \leq -1\}},
\end{equation}
for $1<\gamma<1+\frac{\beta-1}{2}$. After a change of variables we get
\begin{equation}
R_t^T(x)=T^{\frac{1-\gamma}{\beta}}\int_{\frac{1}{T^{\frac{1}{\beta}}}}^\infty|y|^{-\gamma}\left(L_t(y+x)-L_t(-y+x)\right)dy.
\end{equation}
Put
\begin{equation}
Z^{T,\gamma}_t(x):=\int_{\frac{1}{T^{\frac{1}{\beta}}}}^\infty|y|^{-\gamma}\left(L_t(y+x)-L_t(-y+x)\right)dy.
\end{equation}
The above converges almost surely as $T\rightarrow \infty$ to
\begin{equation}\label{Z_def}
Z_t^\gamma(x):= \int_0^\infty |y|^{-\gamma}\left(L_t(y+x)-L_t(-y+x)\right)dy,
\end{equation}
which follows from dominated convergence theorem and Lemma~\ref{fatbound} below.

\begin{lem}\label{fatbound}
Let $Z$ be given by~\eqref{Z_def} and $1<\gamma<1+\frac{\beta-1}{2}$. For $\alpha \in [1,2]$ and $t\geq 0$
\begin{equation}\label{fatwell}
\int_\mathbb{R}\mathbb{E}\left|Z_t^\gamma(x)\right|^\alpha dx <\infty.
\end{equation}
\end{lem}

\begin{proof}
We will show that 
\begin{eqnarray*}
I_1&=&\int_\mathbb{R}\mathbb{E}|Z_t^\gamma(x)|dx <\infty,\\
I_2&=&\int_\mathbb{R}\mathbb{E}|Z_t^\gamma(x)|^2 dx <\infty,
\end{eqnarray*}
which imply~\eqref{fatwell}. Using H\"{o}lder inequality we see that
\begin{equation}
I_1 \leq \int_\mathbb{R}\int_0^\infty |y|^{-\gamma}\left(\mathbb{E}\left(L_t(x+y)-L_t(x-y)\right)^2\right)^{\frac{1}{2}}dydx.
\end{equation}
Observe that (by Lemma~\ref{gbound} in the appendix)
\begin{eqnarray}\label{mom}
\mathbb{E}\left(L_t(x+y)-L_t(x-y)\right)^2&=& 2\int_0^t\int_0^{t-u_1}\Big(p_{u_1}(x+y)+p_{u_1}(x-y)\Big)\nonumber \\
&& \:\Big(p_{u_2}(0)-p_{u_2}(2y)\Big)du_2du_1\nonumber \\
&\leq& 2\Big(\mathbb{E}L_t(x+y)+\mathbb{E}L_t(x-y)\Big)(c_1 \wedge c_2|y|^{\beta-1}),\nonumber
\end{eqnarray}
where the inequality follows from~\eqref{pbound} for some constants $c_1$ and $c_2$ depending only on $\beta$ and $t$. Therefore (using Lemma~\ref{gbound})
\begin{eqnarray*}
I_1&\leq& \int_\mathbb{R}\int_0^\infty |y|^{-\gamma}\left((c_1\wedge c_2|x+y|^{-\frac{\beta+1}{2}})+(c_1\wedge c_2|x-y|^{-\frac{\beta+1}{2}})\right)\\
&& \:(c_3\wedge c_4|y|^{\frac{\beta-1}{2}})dydx,
\end{eqnarray*}
which is finite since $\frac{1+\beta}{2}>1$ and $1<\gamma<1+\frac{\beta-1}{2}$. As for $I_2$, notice that by H\"{o}lder inequality
\begin{eqnarray*}
I_2&=&\int_\mathbb{R}\int_0^\infty \int_0^\infty|y_1|^{-\gamma}|y_2|^{-\gamma}\mathbb{E}\Big((L_t(x+y_1)-L_t(x-y_1))\\
&& \:(L_t(x+y_2)-L_t(x-y_2))\Big)dy_1 dy_2 dx \\
&\leq&
\int_\mathbb{R}\Bigg(\int_0^\infty |y|^{-\gamma}\left(\mathbb{E}\left(L_t(x+y)-L_t(x-y)\right)^2\right)^{\frac{1}{2}}dy\Bigg)^2 dx
\end{eqnarray*}
Seeing that the function
\begin{equation}
x\mapsto \int_0^\infty |y|^{-\gamma}\left(\mathbb{E}\left(L_t(x+y)-L_t(x-y)\right)^2\right)^{\frac{1}{2}}dy
\end{equation}
is bounded uniformly in $x \in \mathbb{R}$, we conclude that since $I_1$ is finite, $I_2$ is finite as well.
\end{proof}

The process $(Z_t^\gamma(x))_{t\geq 0}$ is continuous and has a non-zero mean as long as $x\neq 0$. Using H\"{o}lder inequality it is easy to see that the process $Z^\gamma$ has all moments. If we choose 
\begin{equation}
F_T=T^{1+1/(\alpha \beta)-\gamma/\beta},
\end{equation}
then, we will see (in the more general setting of Theorem~\ref{THM3}) that the finite dimensional distributions of the process $(G_t^T)_{t\geq 0}$ in~\eqref{FC} converge to the finite dimensional distributions of the process which has the integral representation
\begin{equation}\label{stab}
V_t=\int_{\mathbb{R}\times \Omega'}c(\alpha)^{\frac{1}{\alpha}}Z_t(x,\omega')M(dx,d\omega'),
\end{equation}
where $M$ is a symmetric $\alpha$-stable random measure on $\mathbb{R}\times \Omega'$  with intensity $\lambda_1\otimes \mathbb{P'}$ and
\begin{equation}
Z_t(x,\omega')=\int_0^\infty |y|^{-\gamma}\left(L_t(x+y,\omega')-L_t(x-y,\omega')\right)dy,\quad x\in \mathbb{R}.
\end{equation}

\begin{lem}\label{Vprop}
The process $V$ is $H$-sssi with $H=1+\frac{1}{\alpha \beta}-\frac{\gamma}{\beta}$. 
\end{lem}
\begin{proof}
%Recall that for any $c>0$ 
%\begin{equation}\label{local}
%\left(L_{ct}(c^{1/\beta}x)_{x\in \mathbb{R},t \geq 0}\right)\overset{d}{=}c^{1-1/\beta}\left(L_t(x)_{x\in \mathbb{R},t \geq 0}\right).
%\end{equation}
Let $a_1,\ldots,a_m \in \mathbb{R}$ and $0\leq t_1\leq \ldots \leq t_m <\infty$. Take $c>0$ and notice that using Remark~\ref{leqi}
\begin{eqnarray*}
\mathbb{E}\big( exp\Big(i\sum_{j=1}^m a_j G_{t_j}\Big)\big)&=&exp\big(-\int_\mathbb{R}\mathbb{E}\Big|\sum_{j=1}^m a_j Z_{ct_j}(x)\Big|^\alpha dx\big)\\
&=& exp\Big(-c^{1/\beta}\int_\mathbb{R}\mathbb{E}\Big|\sum_{j=1}^m a_j Z_{ct_j}(c^{1/\beta}x)\Big|^\alpha dx\Big)\\
&=& exp\Big(-c^{1/\beta}\int_\mathbb{R}\mathbb{E}\Big|\sum_{j=1}^m a_j \int_0^\infty|y|^{-\gamma}\Big(L_{ct_j}(c^{1/\beta}x+y)\\
&& \:-L_{ct_j}(c^{1/\beta}x-y)\Big)\Big|^\alpha dy dx\Big)\\
&=& exp\Big(-\mathbb{E}\Big|\sum_{j=1}^m a_j \int_0^\infty|y|^{-\gamma}\Big(L_{ct_j}(c^{1/\beta}x+c^{1/\beta}y)\\
&& \:-L_{ct_j}(c^{1/\beta}x-c^{1/\beta}y)\Big)\Big|^\alpha dx  c^{1/\beta}c^{-(\gamma \alpha)/\beta}c^{\alpha/\beta}\Big)\\
&=&\mathbb{E}\left( exp(ic^H\sum_{j=1}^m a_j G_{t_j})\right)
\end{eqnarray*}
where the last inequality follows from~\eqref{leqi}. Hence $G$ is self-similar with Hurst coefficient $H=1+\frac{1}{\alpha \beta}-\frac{\gamma}{\beta}$. The stationarity of increments follows immediately once we notice that $L_{t+s}(z):=L_{t+s}^0(z)=L_s(z)+L_t^{\xi_s}(z)$ for $s,t\geq 0$ and $z \in \mathbb{R}$ and use Fubini theorem.
\end{proof}
\begin{rem}
In this setting $H$ can take any value from the interval $(\frac{1}{2},1)$.
\end{rem}

\subsubsection{Proof of Theorem~\ref{THM3}}

\begin{comment}

What happens if we replace the function
\begin{equation}
y \mapsto |y|^{-\gamma}\mathbf{1}_{\{y\geq 1\}}-|y|^{-\gamma}\mathbf{1}_{\{y \leq -1\}},
\end{equation}
by a more general function (satisfying $\int_\mathbb{R}\phi(y)dy=0$) which decreases relatively "slowly'' to infinity, i.e., at least one of its tails is of the order $|y|^{-\gamma}$ as $y\rightarrow \infty$ or $y\rightarrow -\infty$ for $\gamma \in (1,1+\frac{\beta-1}{2})$? Recalling the setting of Theorem~\ref{THM3}, let $f_1,f_2:(0,\infty)\rightarrow \mathbb{R}$ be regularly varying at infinity with exponents $-\gamma_1$ and $-\gamma_2$ respectively. Let $\gamma_1,\gamma_2>1$ and assume that at least one of the exponents is smaller than $1+\frac{\beta-1}{2}$. Put
\begin{equation}\
\phi(y)=
\begin{cases}
f_1(y), & y>0\\
f_2(-y), & y<0.
\end{cases}
\end{equation}
We have a couple of possible cases. Suppose first that $\gamma_1$ and $\gamma_2$ are different and without loss of generality let $\gamma_1<\gamma_2$. Let $g_1$ and $g_2$ be two functions varying slowly at infinity such that $f_1=|\cdot|^{-\gamma_1}g_1(\cdot)$ and $f_2=|\cdot|^{-\gamma_2}g_2(\cdot)$.

\end{comment}

\begin{proof}[Proof of Theorem~\ref{THM3} in the case $\gamma_1<\gamma_2$]
Without loss of generality we may assume that $f_1$ and $f_2$ are nonnegative. Take any $a_1,\ldots,a_m\in\mathbb{R}$ and $t_1,\ldots,t_m\geq 0$ and recall that the normalization in this case is given by 
\begin{equation*}
F_T=g_1(T^{1/\beta})T^{1+1/(\alpha \beta)-\gamma/\beta}.
\end{equation*}
Then for $G^T$ as in~\eqref{FC} we have (after a change of variables and using symmetry)
\begin{eqnarray}\label{cloc1}
\mathbb{E}exp\Big(\sum_{j=1}^m a_j G_{t_j}^T\Big) &=&exp\Bigg(\int_\mathbb{R}\int_\mathbb{R}\mathbb{E}\Big(e^{iM^T(x)}-1-\mathbf{1}_{\{|z|\leq 1\}}zM^T(x)\Big)\nonumber\\
&& \:\times\mathbf{1}_{\{|z|>T^{-1/\alpha \beta}\}}|z|^{-1-\alpha}dzdx\Bigg),
\end{eqnarray}
where
\begin{multline}
M^T(x)=\sum_{j=1}^m a_j \Big(g(T^{1/\beta})^{-1}\int_0^\infty g(T^{1/\beta}y)|y|^{-\gamma_1}\big(L_{t_j}(x+y)-L_{t_j}(x)\big)dy\\
-T^{\frac{-\gamma_2+\gamma_1}{\beta}}g_1(T^{1/\beta})^{-1}\int_0^\infty|y|^{-\gamma_2}g_2(T^{1/\beta}y)\big(L_{t_j}(x-y)-L_{t_j}(x)\big)dy \Big).
\end{multline}
By Lemma~\ref{fatbound}, in the limit of~\eqref{cloc1} we can forget about the term \\
$\mathbf{1}_{\{|z|>T^{-1/\alpha \beta}\}}$ and by Lemma~\ref{trick} in the Appendix we only need to show the convergence of
\begin{equation}\label{cloc2}
exp\Big(\int_\mathbb{R}\mathbb{E}|M^T(x)|^\alpha dx \Big).
\end{equation}
Very similarly as in the proof of Theorem~\ref{THM2} on can show that 
\begin{equation}
\lim_{K\rightarrow \infty} \limsup_{T\rightarrow \infty}\int_{|x|>K}\mathbb{E}|M^T(x)|^\alpha dx = 0,
\end{equation} 
and we will skip the proof. Thus, it remains to show that for any $x\in \mathbb{R}$, $\mathbb{E}|M^T(x)|^\alpha$, converges, up to multiplicative constant to
\begin{equation}
\mathbb{E}\Big|\sum_{j=1}^m a_j \widetilde{Z}_{t_j}\Big|^\alpha,
\end{equation}
where $\widetilde{Z}$ is defined by~\eqref{Z2}. For simplicity let us assume that $m=1$, $a_1=1$ and $t_1=t$. We can write $M^T(x) = M^{T,1}(x) - M^{T,2}(x)$ with
\begin{eqnarray}
M^{T,1}(x)&=&\int_0^\infty \frac{g(T^{1/\beta}y)}{g(T^{1/\beta})}|y|^{-\gamma_1}\big(L_t(x+y)-L_t(x)\big)dy,\label{cloc3}\\
M^{T,1}(x)&=&T^{\frac{-\gamma_2+\gamma_1}{\beta}}g_1(T^{1/\beta})^{-1}\int_0^\infty|y|^{-\gamma_2}g_2(T^{1/\beta}y)\nonumber\\
&& \: \times\big(L_t(x-y)-L_t(x)\big)dy.\label{cloc4}
\end{eqnarray}
Fix $r\in (0,1)$. Similarly as in, e.g., Lemma~\ref{clim}, on can show, using Lemma~\ref{fatbound}, that
\begin{multline}
\lim_{r\rightarrow 0_+}\limsup_{T\rightarrow \infty}\mathbb{E}\Big|g(T^{1/\beta})^{-1}\int_0^r g(T^{1/\beta}y)|y|^{-\gamma_1}\big(L_t(x+y)-L_t(x)\big)dy\Big|=0
\end{multline}
We can show that the equivalent holds for $M^{T,2}(x)$. Now, as we consider only $y$ bounded away from zero, we can use Theorem 10.5.5 and Corollary 10.5.8 in~\cite{SPLRD} to bound the integrals in $M^{T,1}(x)$ and $M^{T,2}(x)$ uniformly for all $T$ large enough and use dominated convergence to show that
\begin{equation}
\mathbb{E}|M^T(x)|^\alpha \rightarrow \mathbb{E}\Big|\widetilde{Z}_{t}\Big|^\alpha
\end{equation}
as $T\rightarrow\infty$. This finishes the proof.

The proofs of other cases of Theorem~\eqref{THM3} are very much like the one above and we skip them for the sake of brevity.
\end{proof}

\begin{appendices}

\section{Preliminary Properties of Stable Local Times}

Let $(\xi_t)_{t \geq 0}$ be a symmetric $\beta$-stable Levy process with $\beta \in (1,2)$. It is well known that in this case $\xi$ admits a jointly continuous local time. Denote it by $L_t(x)$, $x\in \mathbb{R}, t\geq 0$. Here we provide number of facts that are used throughout this paper. their proofs are relatively straightforward and we skip them and provide the necessary references.

\begin{lem}\label{lmoments}
Let $(L_t(x))_{t\geq 0}$ be a local time at $x \in \mathbb{R}$ of a symmetric $\beta$-stable process (denoted by $\xi$) with $\beta \in (1,2)$. Then for any $n \in \mathbb{N}$ and $t>0$
\begin{multline}
\mathbb{E}\left(L_t(x)\right)^n=\\
n!\frac{1}{(2\pi)^n}\int_0^{t}\ldots \int_{u_{n-1}}^t\int_{\mathbb{R}^n}e^{ixz_1}e^{-(u_n-u_{n-1})|z_n|^\beta}e^{-(u_{n-1}-u_{n-2})|z_{n-1}|^\beta}\ldots e^{-u_1|z_1|^\beta}\\
dz_1\ldots dz_n d{u_1}\ldots du_n.
\end{multline}

We also have
\begin{multline}
\mathbb{E}\left(L_t(x)\right)^n=\\
n!\int_0^t \int_{u_1}^t \ldots \int_{u_{n-1}}^t p_{u_n-u_{n-1}}(0)\ldots p_{u_2-u_1}(0)p_{u_1}(x)du_n\ldots du_1,
\end{multline}
and
\begin{multline}\label{mmoments}
\mathbb{E}L_t(x_1)\ldots L_t(x_n)= \sum_{\pi \in \Pi(n)}\int_{\Delta^n_T}p_{u_n-u_{n-1}}(x_{\pi_n}-x_{\pi_{n-1}})\ldots\\
\ldots p_{u_2-u_1}(x_{\pi_2}-x_{\pi_{2}})p_{u_1}(x_{\pi_1})du_n\ldots du_1.
\end{multline}
The proof is very similar to the proof of Lemma 1 in~\cite{ROS} and we skip it.
\end{lem}

We will need a lemma about the asymptotic behavior of $\mathbb{E}L_t(x)$ as $|x|\rightarrow \infty$. The proof is straightforward so we skip it.
\begin{lem}\label{gbound}
For any $t>0$ there exist constants $C,C'$ depending only on $t$ and $\beta$ such that 
\begin{equation}
\mathbb{E}L_t(x) \leq C\wedge (C'|x|^{-\beta-1}).
\end{equation}
\end{lem}
We also have the following lemma.

\begin{lem}
For $x\in \mathbb{R}$, $x\neq 0$:
\begin{multline}\label{cint}
\int_0^t\big(p_u(0)-p_u(x))du=|x|^{\beta-1}\int_0^{t|x|^{-\beta}}\Big(p_1(0)-p_1\big(\frac{1}{u^{1/\beta}}\big)\Big)\frac{1}{u^{1/\beta}}du.
\end{multline}
Here $p$ is the $\beta$-stable transition density. Putting 
\begin{equation}\label{cdef}
c=\int_0^\infty \Big(p_1(0)-p_1\big(\frac{1}{u^{1/\beta}}\big)\Big)\frac{1}{u^{1/\beta}}du,
\end{equation}
which is finite (see~\cite{ROS}), and noticing that the first integral in~\eqref{cint} is bounded by a constant $c_1$ depending only on $t$ and $\beta$, we conclude that
\begin{equation}\label{pbound}
\int_0^t\big(p_u(0)-p_u(x))du \leq c_1\wedge c|x|^{\beta-1}, \quad x\neq 0.
\end{equation}
\end{lem}

Lemmas~\ref{localbound}, \ref{localboundmix} and \ref{lbound} below are consequences of Lemma~\ref{gbound} and the fact that $\int_0^t p_u(x)dy = \mathbb{E}L_t(x)$.

\begin{lem}\label{localbound}
For any $t>0$ and any positive $p>0$
\begin{equation}
\mathbb{E}|L_t(x)|^p < \infty,
\end{equation}
uniformly in $x \in \mathbb{R}$.
\end{lem}

\begin{rem}\label{localboundmix}
Using H\"older inequality and Lemma~\ref{localbound} we see that 
\begin{equation}
\mathbb{E}|L_t(x_1)\ldots L_t(x_m)| < \infty 
\end{equation}
uniformly in $x_1,\ldots,x_m \in \mathbb{R}$.
\end{rem}
\begin{lem}\label{lc}
For any $p \in (0,\infty)$ and $\beta \in (1,2)$ we have
\begin{equation}\label{lbound}
\int_\mathbb{R}\mathbb{E}\left|L_t(x)\right|^p dx <\infty.
\end{equation}
\end{lem}

\begin{rem}\label{leqi}
For any $a>0$ and $z\in \mathbb{R}$ the process $(L_{ct}(z))_{t\geq 0}$ has the same law as $(c^{1-1/\beta}L_t(\frac{z}{c^{1/\beta}}))_{t\geq 0}$ (see Proposition 10.4.8 in~\cite{SPLRD}).
\end{rem}

\begin{lem}\label{trick}
For $\alpha \in (0,2)$, $x\in \mathbb{R}$ and any $R>0$ we have
\begin{equation}
c(\alpha)|x|^\alpha=\int_\mathbb{R}\left(1-e^{ixu}+ixu\mathbf{1}_{\{|u|\leq R\}}\right)\frac{du}{|u|^{1+\alpha}},
\end{equation}
where $c(\alpha)$ is a constant independent of $x$ and $R$.
\end{lem}

\end{appendices}

\bibliographystyle{plain}

\bibliography{publications}

\end{document}